\documentclass[12pt,reqno]{amsart}
\usepackage{cases}
\usepackage{amscd}
\usepackage{amsfonts}
\usepackage{amssymb}
\usepackage{amsmath}
\usepackage{graphicx}
\usepackage{epstopdf}
\usepackage{subfigure}
\usepackage{amsthm}
\usepackage{stmaryrd}

\theoremstyle{remark}
\newtheorem{example}{\textbf{Example}}[section]
\numberwithin{equation}{section}

\usepackage{color}
\usepackage{datetime}

\def\tr{\textcolor{red}}

\makeatletter
\newcommand\figcaption{\def\@captype{figure}\caption}
\newcommand\tabcaption{\def\@captype{table}\caption}
\makeatother

\usepackage{geometry}
\usepackage{algorithm}
\usepackage{multirow}

\def\bq{\begin{equation}}
\def\eq{\end{equation}}
\def\bqs{\begin{equation*}}
\def\eqs{\end{equation*}}
\def\bsqs{\begin{subequations}}
\def\esqs{\end{subequations}}
\def\ba{\begin{aligned}}
\def\ea{\end{aligned}}
\def\br{\begin{eqnarray}}
\def\er{\end{eqnarray}}
\def\brr{\bq\begin{array}{rlll}}
\def\err{\end{array}\eq}

\def\text#1{\hbox{#1}}
\newtheorem{thm}{Theorem}[section]
\newtheorem{lem}{Lemma}[section]

\newtheorem{rem}{Remark}[section]

\newcommand{\bsub}{\begin{subequations}}
\newcommand{\esub}{\end{subequations}$\!$}

\title[IEQ-DG methods for the Cahn-Hilliard equation]{Unconditionally energy stable discontinuous Galerkin schemes for the Cahn-Hilliard equation}

\author[H.~Liu, P. ~Yin]{Hailiang Liu$^\dagger$ and Peimeng Yin$^\S$}
\address{$^\ddagger$ Iowa State University, Department of Mathematics, Ames, IA 50011} \email{hliu@iastate.edu}
\address{$^\S$ Wayne State University, Department of Mathematics, Detroit, MI 48202} \email{pyin@wayne.edu}

\begin{document}

\begin{abstract}
In this paper, we introduce novel discontinuous Galerkin (DG) schemes for the Cahn-Hilliard equation,
which arises in many applications.
The method is designed by integrating the mixed DG method for the spatial discretization with the \emph{Invariant Energy Quadratization}  (IEQ) approach for the time discretization.
Coupled with a spatial projection, the resulting IEQ-DG schemes are shown to be unconditionally energy dissipative, and can be efficiently solved without resorting to any iteration method. 
Both one and two dimensional numerical examples are provided to verify the theoretical results,  and demonstrate the good performance of IEQ-DG in terms of  efficiency, accuracy, and preservation of the desired solution properties. 
\end{abstract}

\maketitle

\bigskip


\section{Introduction}
The Cahn-Hilliard (CH) equation, originally introduced in \cite{CH58} as a model of phase separation in binary alloys, has become a fundamental equation as well as a building block in the phase field methodology for moving interface problems arising from various applications (see, e.g., \cite{Mc02} for the references therein).

This work is concerned with high order numerical approximations to the Cahn-Hilliard problem:  find $\{u(x, t), w(x, t)\}$
for $x\in \Omega$ and $t>0$ such that
\begin{equation}\label{CH}
\begin{aligned}
 u_t & = \nabla \cdot (M(u) \nabla w),  \\ 
 w & = - \epsilon^2 \Delta u + F'(u), \\ 
 u(x,0) & = u_0(x),  \quad x \in \Omega,
\end{aligned}
\end{equation}
where $\Omega \subseteq \mathbf{R}^d(d=1,2, 3)$ is a bounded domain, $\epsilon$ is a positive parameter, $M(u)\geq 0$ is the mobility function, $F(u)$ is the nonlinear bulk potential, and $u_0(x)$ is the initial data.

The well-posedness study of the Cahn--Hilliard equation is very rich, and the existing results are mainly  for two types of models:  one is the constant mobility with polynomial potentials (\cite{EZ86}), and another is the degenerate mobility of form $M(u)=u(1-u)$  (\cite{EG96}) with the potential of logarithmic types (see, e.g., \cite{CH58, EL91,Y92, EG96, DD95}). Here we numerically study model (\ref{CH}) with no restrictions on the specific form of the mobility and free energy.  For the degenerate CH model, we apply the established scheme to regularized systems as discussed in Section \ref{sec4}. While we also refer the reader to 
  \cite{PP19} for a new relaxation system to approximate the degenerate CH model.

We consider in this paper the following boundary conditions:
\begin{equation}\label{BC}
\text{(i)} \ u \text{ is periodic};  \quad \text{or (ii)}   \ \partial_\mathbf{n} u=M(u)\partial_\mathbf{n} w=0, \quad x\in \partial \Omega.
\end{equation}
Here $\mathbf{n}$ stands for the unit outward normal to the boundary $\partial \Omega$. With such boundary conditions,  equation (\ref{CH}) is endowed with  a gradient flow structure $u_t= \nabla \cdot(M(u)\frac{\delta}{\delta u} \mathcal{E})$, dictated  by the energy dissipation law
\begin{align}\label{di}
\frac{d}{dt} \mathcal{E}(u) = - \int_\Omega M(u) |\nabla w|^2 \leq 0,
\end{align}
where the total free energy is defined by
\bq\label{ch_engdisO}
\mathcal{E}(u) = \int_{\Omega} \left( \frac{\epsilon^2}{2}|\nabla u|^2 + F(u) \right)dx.
\eq
The CH model is nonlinear and its analytical solutions are intractable.  Also, for a gradient flow such as the CH equation steady states are of particular interest.  
Hence, designing accurate, efficient, and energy stable algorithms to solve it becomes essential, in particular
for the accuracy of long time simulations. Keeping the energy dissipation (\ref{di})
has been a major concern in the design of various numerical schemes, see, e.g. \cite{BE92, Ey98, F01, S04, SY10, FKU17, SXY19, YZ19,CWWW19}.

In this paper, we aim to develop unconditionally energy stable discontinuous Galerkin (DG) schemes for solving the CH model problem. To achieve this, we face two main challenges: (i) how to handle fourth order derivatives in the DG discretization; and (ii) how to handle the nonlinear term associated with the potential $F$ in time discretization.

For (i), several approaches have been adopted to deal with difficulties caused by the high order solution derivatives.  The first one is  the local DG (LDG) methods \cite{XXS07, GX14, SS17}, with which the original equation is rewritten into a first order system for further DG discretization. The second one is the mixed symmetric interior penalty DG (SIPG) methods \cite{WKG06, FK07, FLX16, FKU17}, with which the penalty terms are added as interface corrections upon the global solution formulation so that the resulting scheme is stable. 
 It is also possible to apply an ultra-weak DG discretization, such as the DG scheme in \cite{CS08} for the one-dimensional biharmonic equation.
The spatial DG discretization we adopt here takes advantages of both the SIPG method and the
mixed DG method without interior penalty in \cite{LY18, LY19} for another class of  fourth order PDEs.


For (ii), there are several related time discretization techniques available in the literature, including the so-called convex splitting approach \cite{Ey98, WZZ14, CWWW19} and the stabilization approach \cite{SY10, XT06}.  The convex splitting approach, introduced in the pioneering work \cite{Ey98}, treats the convex part of the potential implicitly and the concave part explicitly.  As a result such method is unconditionally energy stable, however, it produces nonlinear schemes, thus the implementations are often complicated with potentially high computational costs. The stabilization approach,  introduced in \cite{XT06} for epitaxial growth models,  treats the nonlinear term explicitly so that the resulting scheme becomes stable after a stabilization term is added to avoid strict time step constraints. However, a large stabilization constant is often needed to ensure energy stability.

A more feasible remedy would be to make a reasonable transform of the given potential.  One technique  is  the \emph{Invariant Energy Quadratization} (IEQ) approach introduced in \cite{Y16, ZWY17}, which generalized the two types of  linear energy stable schemes introduced in \cite{GT13}.  The IEQ method is known to provide flexibilities to treat the nonlinear terms since it only requests that the nonlinear potential be bounded from below. We note that recently a related approach, called the SAV method, has been introduced in \cite{SY18, SXY19}  with certain advantages over the IEQ. Yet  efficiently solving  the resulting linear system when coupled with a DG spatial discretization is subtle, see \cite{LY20}.
%

In order to design IEQ-DG schemes,  our strategy is to start with the model satisfying two basic assumptions: \\
(i) the mobility function $M(u)$ satisfies
$$
M(u) \geq M_{\min}>0;
$$
(ii) there exist a constant $B>0$, such that
$$
F(u)>-B,
$$
for any $u$ under consideration, and  further discuss how to extend the established schemes to more general cases.

The IEQ-DG method introduced in \cite{LY19} for  the Swift--Hohenberg equation has several advantages,  such as high order of accuracy, easy implementation, and efficiency.  Unfortunately, the DG discretization in  \cite{LY19}
when applied to the CH equation (\ref{CH}) does not seem to yield the desired energy stability. Therefore,  in this paper, we exploit the direct DG (DDG) method \cite{LY10} for both $u$ and $w$, which when taking  a special choice of the flux parameters can be reformulated into a mixed symmetric interior penalty DG (SIPG) scheme (see, e.g., \cite{FLX16}).

The IEQ approach for time discretization relies on an auxiliary variable $U=\sqrt{F(u)+B}$, so that
$$
 U_t = \frac{1}{2} H(u)u_t,\quad H(u):=F'(u)/\sqrt{F(u)+B}.
$$
With this transformation we update $U^{n}$ in two steps:  the piecewise $L^2$ projection with $U_h^n=\Pi U^n$, and the update step with
$$
\frac{U^{n+1} - U_h^n}{\Delta t}  = \frac{1}{2} H(u^{n}_h) \frac{u_h^{n+1} - u_h^n}{\Delta t}.
$$
The resulting IEQ-DG scheme follows from replacing the nonlinear function $F'(u_h^{n+1})$ by $H(u^{n}_h)U^{n+1}$ in the DG discretization (see the scheme formulation in (\ref{ch_FPDGFull1st+})).
In addition, the resulting discrete systems are linear. 
As a result, the methods are simple to implement and computationally efficient. With the techniques exploited in \cite{LY21}, the IEQ-DG schemes may be 6
extended to DG schemes of arbitrarily high order in time, yet still unconditionally energy stable.
Finally, closest to our work is \cite{GX19}, where the authors, building on the IEQ formulation with the LDG spatial discretization for phase field problems including the CH equation, proposed energy stable linear schemes combining with the semi-implicit spectral deferred correction to gain higher order time discretization, while the auxiliary variable is computed coupling with other unknowns.  In contrast, our algorithms enable an explicit update for the auxiliary variable, hence more efficient in computation.

\subsection{Our Contributions}
\begin{itemize}
\item We propose to solve (\ref{CH}) by simple IEQ-DG schemes, which integrate
the mixed DG method for spatial discretization with the IEQ approach for time discretization,
coupled with a crucial spatial projection.
\item We show that the semi-discrete DG scheme features a discrete energy dissipation law if the penalty parameter is suitably large,  and present both first and second order (in time) IEQ-DG algorithms. We prove that the IEQ-DG schemes are indeed unconditionally energy stable.
\item We conduct extensive experiments to evaluate the performance of IEQ-DG. First, we present numerical results to show the high order of spatial accuracy of the proposed schemes, the energy dissipating and mass conservative properties of numerical solutions. Second, we conduct experiments on some two-dimensional pattern formation problems, all of which demonstrate the good performance of IEQ-DG.
\end{itemize}

\subsection{Organization} In Section 2, we formulate a unified semi-discrete DG method for the CH equation (\ref{CH}) subject to  different boundary conditions. In Section 3, we present fully discrete DG schemes and show the energy dissipation and mass conservation properties. In Section 4, we discuss extensions to the case with degenerate mobility and the logarithmic Flory-Huggins potential.  In Section 5, we numerically verify the performance of IEQ-DG on different numerical examples. 
Finally in Section 6 some concluding remarks are given.

\section{Spatial DG discretization}

Let the domain $\Omega$ be a union of  rectangular meshes $\mathcal{T}_h=\{K\}:=\bigcup_{\alpha=1}^\mathcal{N} K_\alpha$, with $\alpha=(\alpha_1, \cdots, \alpha_d), \ \mathcal{N}=(\mathcal{N}_1, \cdots, \mathcal{N}_d)$ and $K_\alpha=I_{\alpha_1}^1\times \cdots \times I_{\alpha_d}^d$, where $I_{\alpha_i}^i=[x_{\alpha_i-1/2}^i, x_{\alpha_i+1/2}^i]$ for $\alpha_i=1, \cdots, \mathcal{N}_i$.
Denote by $h_i=\max_{1\leq \alpha_i \leq N_i} |I_{\alpha_i}^i|$, with $h=\max_{1\leq i \leq d}h_i$.
We denote the set of the interior interfaces by $\Gamma^0$, the set of all boundary faces by $\Gamma^\partial$, and $\Gamma_h=\Gamma^0 \cup \Gamma^\partial$.

The discontinuous Galerkin finite element space can be formulated as
$$
V_h = \{v\in L^2(\Omega): \ v|_{K} \in P^k(K), \ \forall K \in \mathcal{T}_h \},
$$
where $P^k(K)$ denotes the set of polynomials of degree no more than $k$ on element $K$.
Let $K_1$ and $K_2$ be two neighboring cells.
If the unit normal vector $\nu$ on element interfaces $e\in \partial K_1 \cap \partial K_2$ is oriented from $K_1$ to $K_2$,
then the average $\{\cdot\}$ and the jump $[\cdot]$ operator are defined by
$$
\{v\} = \frac{1}{2}(v|_{\partial K_1}+v|_{\partial K_2}), \quad [v]=v|_{\partial K_2}-v|_{\partial K_1}, 
$$
for any function $v \in V_h$,  where $v|_{\partial K_i} \ (i=1,2)$ is the trace of $v$ on $e$ evaluated from element $K_i$.

The direct DG discretization of (\ref{CH}) is to find $(u_h(\cdot, t), w_h(\cdot, t)) \in V_h \times V_h$ such that for all $\phi, \ \psi \in V_h$ and $K\in \mathcal{T}_h$
\begin{subequations}\label{SemiDGcell}
\begin{align}
\int_K u_{ht}  \phi dx = &-\int_K M(u_h) \nabla w_h \cdot \nabla \phi dx +  \int_{\partial K} M(\widehat{u_h})\left(\widehat{\partial_\nu w_h} \phi + (w_h-\widehat{w_h})\partial_\nu \phi\right) ds,\\
\int_K w_h  \psi dx = & \epsilon^2\left(\int_K \nabla u_h \cdot \nabla \psi dx -   \int_{\partial K} \widehat{\partial_\nu u_h} \psi + (u_h-\widehat{u_h})\partial_\nu \psi ds \right) + \int_K F'(u_h) \psi dx,
\end{align}
\end{subequations}
where $\nu$ stands for the outward normal direction to $\partial K$.  On each cell interface $e \in \partial K \bigcap \Gamma^0$, the numerical flux is taken as
\begin{equation}\label{ch_fluxI}
\begin{aligned}
\widehat{\partial_\nu v} = \frac{\beta_0[v]}{h_e}+\{\partial_\nu  v\}, \ \widehat{v} = \{v\},
\end{aligned}
\end{equation}
for $v=w_h, \ u_h$, where $\beta_0>0$ is a parameter to be determined. Here $h_e$ is the characteristic length of interface $e$. In case of the uniform meshes, we take $h_e=h_i$ at each interface $x_{\alpha_i+1/2}^i$ for $\alpha_i=0, 1, \cdots, \mathcal{N}_i$. The numerical fluxes on $e \in \partial K \bigcap \Gamma^\partial$ depend on the boundary conditions.
For periodic boundary conditions, the numerical fluxes can take the same formula as those in (\ref{ch_fluxI}). For homogeneous Neumann boundary conditions,
the numerical fluxes on the boundary $e \in \partial K \bigcap \Gamma^\partial$ are defined as
\begin{equation}\label{fluxN}
\widehat{\partial_\nu w_h} = 0, \ \widehat{w_h} = w_h, \ \widehat{\partial_\nu u_h} = 0, \ \widehat{u_h} = u_h.
\end{equation}
Summation of (\ref{SemiDGcell}) over all elements $K\in \mathcal{T}_h$ leads to a global DG formulation
\begin{subequations}\label{ch_SemiDGN}
\begin{align}
(u_{ht},\phi) = &- A(M(u_h); w_h,\phi),\\
(w_h, \psi) = &   A(\epsilon^2; u_h,\psi)+\left(F'(u_h),\psi \right),
\end{align}
\end{subequations}
where the bilinear functional is given by
\bq \label{ch_A0}
\ba
A(a(x);q,v) = & \sum_{K\in \mathcal{T}_h} \int_K a(x) \nabla q \cdot \nabla v dx + \sum_{e\in \Gamma^0} \int_e a(x)\left(\widehat{ \partial_\nu q} [v]+[q]\{ \partial_\nu v\} \right)ds\\
& +\frac{\tau}{2} \sum_{e\in \Gamma^\partial} \int_{e} a(x) \left( \widehat{ \partial_\nu q} [v]+ [q]\{\partial_\nu v\} \right)ds,
\ea
\eq
where $\tau=1$ for (i) in (\ref{BC}) and $\tau=0$ for (ii) in (\ref{BC}).
Here $a(x)=M(u_h)$ in $K$ but $M(\widehat{u_h})$ for $x \in e$. Note that the factor $\frac{\tau}{2}$ in (\ref{ch_A0}) is used to indicate that for periodic boundary conditions only one end in each direction should be counted. Here each respective type of boundary conditions specified in  (\ref{BC}) has been taken into account.  The initial data for $u_h$ is taken from $V_h$ so that
\bqs
\int_{\Omega} (u_0(x)-u_h(x,0))\phi dx=0, \quad \forall \phi\in V_h.
\eqs
As usual we denote $u_h(x,0)=\Pi u_0(x)$, where $\Pi$ is the piecewise $L^2$ projection.\\

We introduce the discrete energy
\bq\label{ch_oeq}
E(u_h)=\frac{1}{2}A(\epsilon^2; u_h,u_h)+\int_{\Omega}F(u_h) dx,
\eq
and the notation
\bq\label{DGnorm}
\|v\|^2_{DG}:= \sum_{K\in \mathcal{T}_h} \int_K |\nabla v|^2 dx + \left(\sum_{e\in \Gamma^0} +\frac{\tau}{2} \sum_{e\in \Gamma^\partial} \right) \int_e \frac{\beta_0}{h_e}[v]^2ds, \quad \forall v \in V_h.
\eq
We can show that if $\beta_0$ is suitably large, the semi-discrete DG scheme (\ref{ch_SemiDGN})
features a discrete energy dissipation law. 


\begin{lem}
For piecewise continuous function $a(x)>0$, there exists $\beta^*> 0$ such that if $\beta_0>\beta^*$, then
\bq\label{VCoer}
A(a(x); v,v) \geq c_0 \inf_{x\in \Omega}a(x)  \|v\|_{\rm DG}^2  \quad \forall v\in V_h. 
\eq
for some $c_0>0$.  As a result, we have
\bq\label{semieng}
\frac{d}{dt}E(u_h) = -A(M(u_h); w_h,w_h) \leq 0, \quad \forall t>0.
\eq
\end{lem}
\begin{proof}
(i)
By Young's inequality, we have
\bqs
\ba
A(a(x);v,v)= & \sum_{K\in \mathcal{T}_h} \int_K a(x) |\nabla v|^2 dx + \left(\sum_{e\in \Gamma^0} + \frac{\tau}{2} \sum_{e\in \Gamma^\partial} \right) \int_e a(x)[v]\left( \frac{\beta_0}{h_e} [v]+2\{ \partial_\nu v\} \right)ds \\
\geq & \sum_{K\in \mathcal{T}_h} \int_K a(x) |\nabla v|^2 dx + \left(\sum_{e\in \Gamma^0} + \frac{\tau}{2} \sum_{e\in \Gamma^\partial} \right) \int_e a(x) \frac{\beta_0}{h_e}[v]^2ds  \\
& - \left(\sum_{e\in \Gamma^0} + \frac{\tau}{2} \sum_{e\in \Gamma^\partial} \right) \left( \frac{\alpha}{h_e} \int_e a(x)[v]^2ds + \frac{h_e}{\alpha} \int_e a(x)\{ \partial_\nu v\}^2ds\right)\\
=&  \sum_{K\in \mathcal{T}_h} \int_K a(x) |\nabla v|^2 dx - \frac{1}{\alpha}\left(\sum_{e\in \Gamma^0} + \frac{\tau}{2} \sum_{e\in \Gamma^\partial} \right)h_e\int_e a(x) \{ \partial_\nu v\}^2ds\\
& + \left(\sum_{e\in \Gamma^0} + \frac{\tau}{2} \sum_{e\in \Gamma^\partial} \right) \left( \frac{\beta_0 -\alpha }{h_e} \right)\int_e a(x)[v]^2ds
\ea
\eqs
for $0 < \alpha < \beta_0$.

Set
\bq\label{beta0}
\beta^* = \sup_{v \in \tilde V_h}\frac{\left(\sum_{e\in \Gamma^0} + \frac{\tau}{2} \sum_{e\in \Gamma^\partial} \right)h_e\int_e a(x) \{ \partial_\nu v\}^2ds}{\sum_{K\in \mathcal{T}_h} \int_K a(x) |\nabla v|^2 dx},
\eq
where $\tilde V_h:=\{v\in V_h |\quad v\not= \text{const}\}$. For given $a(x)$ positive, this quantity $\beta^*$ can be shown to be uniformly bounded.  In fact, it suffices to bound each local ratio by
$$
\frac{h_e\int_{\partial K} a(x) |\partial_\nu v|^2ds}{\int_K a(x) |\nabla v|^2 dx}
\leq \frac{\max_{x\in \partial K}  a(x)}{\min_{x\in K} a(x)} \frac{h_e \int_{\partial K}|\partial_\nu v|^2ds}{\int_K |\nabla v|^2 dx},
$$
in which the ratio without weight $a$ is known to be bounded by a constant depending only on $k$; see  \cite[Lemma 3.1]{L15}. Thus we have
\bqs
\ba
A(a(x);v,v) \geq  \left( 1- \frac{\beta^*}{\alpha}\right) \sum_{K\in \mathcal{T}_h} \int_K a(x) |\nabla v|^2 dx + \left( 1-\frac{\alpha}{\beta_0}\right)\left(\sum_{e\in \Gamma^0} + \frac{\tau}{2} \sum_{e\in \Gamma^\partial} \right)\int_e a(x)\frac{\beta_0}{h_e}[v]^2ds.
\ea
\eqs
Set
$\alpha = \sqrt{\beta_0\beta^*}$ and  $c_0=\left(1-\sqrt{\frac{\beta^*}{\beta_0}}\right)$, we obtain
$$
A(a(x); v,v) \geq c_0 \inf_{x\in \Omega}a(x) \|v\|_{\rm DG}^2
$$
for $\beta_0>\beta^*$.

(ii) Taking $\phi=w_h, \ \psi=u_{ht}$ in (\ref{ch_SemiDGN}), then (\ref{semieng}) follows immediately.
\end{proof}
\begin{rem} Here $\beta^*$ clearly depends on the choice of $a(x)$ unless it is a constant,  and the type of boundary conditions through $\tau$. The fact that $\beta^*$ is increasing in $\tau$ implies that for Neumann  boundary conditions it suffices to take a smaller $\beta^*$.
\end{rem}

\section{Time discretization}
\subsection{The IEQ reformulation}
The basic idea of the IEQ methodology \cite{YZWS17, YZ19} is to rewrite the energy functional into a quadratic form
\bq\label{eng}
E(u_h, U)=\frac{1}{2}A(\epsilon^2; u_h,u_h)+\int_{\Omega}U^2dx  =E(u_h)+B|\Omega|,
\eq
where
$$
U=\sqrt{F(u_h)+B}
$$
is well-defined when $B$ is chosen so that $F(u_h)+B>0$.
With the IEQ approach $U$ is updated by solving
$$
U_t =  \frac{1}{2}H(u_h) u_{ht},
$$
where
\begin{align}\label{ch_hw}
 H(w)= \frac{F'(w)}{\sqrt{F(w)+B}}.
\end{align}

For the semi-discrete DG scheme (\ref{ch_SemiDGN}), we follow \cite{LY19} where an IEQ-DG method was developed for the Swift-Hohenberg equation,
to consider the following system: find $(u_h, w_h) \in V_h  \times V_h $ and $U(x, t)$ such that
\begin{subequations}\label{SemiDG+}
\begin{align}
U_{t} =&  \frac{1}{2}H(u_h) u_{ht},\\
(u_{ht}, \phi) = & A(M(u_h); w_h,\phi) ,\\
(w_h, \psi) = & A(\epsilon^2; u_h,\psi)+\left(H(u_h)U,\psi \right),
\end{align}
\end{subequations}
for all $\phi, \psi \in V_h $, subject to initial data
$$
U(x, 0)=\sqrt{F(u_0(x))+B}, \quad u_h(x, 0)=\Pi u_0(x).
$$
Note that with the modified discrete energy (\ref{eng}) we still have the following
\bqs
\frac{d}{dt}E(u_h, U) = -A(M(u_h); w_h,w_h) \leq 0.
\eqs
We proceed to discretize (\ref{SemiDG+}) in time. 

\subsection{First order fully discrete IEQ-DG scheme}
Find  $(u^{n}_h, w_h^{n}) \in V_h  \times V_h $ and $U^n=U^n(x)$ such that
\begin{subequations}\label{ch_FPDGFull1st+}
\begin{align}
U^n_h= & \Pi U^n, \\
\frac{U^{n+1} - U_h^n}{\Delta t} = & \frac{1}{2}H(u_h^n) \frac{u_h^{n+1} - u_h^n}{\Delta t},\\
\left(  \frac{u_h^{n+1} - u_h^n}{\Delta t}, \phi \right)= & -A(M(u_h^n); w_h^{n+1}, \phi), \\
(w_h^{n+1},\psi) = & A(\epsilon^2; u_h^{n+1}, \psi)+\left( H(u_h^n)U^{n+1}, \psi \right),
\end{align}
\end{subequations}
for $\phi, \psi \in V_h $.

This scheme admits the following properties.
\begin{thm}\label{ch_firstorder+} Suppose that $ M_{\min} \leq M(u)\leq M_{\max}$ for $u$ under consideration, and set
\bq\label{beta0+}
\beta_0^*=\frac{M_{\max}}{M_{\min}} \sup_{v \in \tilde V_h}\frac{\left(\sum_{e\in \Gamma^0} + \frac{\tau}{2} \sum_{e\in \Gamma^\partial} \right)h_e\int_e  \{ \partial_\nu v\}^2ds}{\sum_{K\in \mathcal{T}_h} \int_K |\nabla v|^2 dx}.
\eq
If $\beta_0> \beta_0^*$,  scheme (\ref{ch_FPDGFull1st+}) admits a unique solution $(u_h^{n}, w_h^{n})$ for any $\Delta t>0$, and the solution $u_h^n$ satisfies the mass conservation, i.e.,
\bq\label{ch_ConMass}
\int_\Omega u_h^{n} dx = \int_\Omega u_h^{0} dx,
\eq
for any $n>0$.  Moreover, for $E^n := E(u_h^n, U_h^n)$ we have
\bq
\ba\label{ch_engdis1st+}
E^{n+1} \leq E(u_h^{n+1}, U^{n+1}) = &  E^n - \Delta t  A(M(u_h^n);w_h^{n+1},w_h^{n+1})\\
& -\frac{1}{2}A(\epsilon^2; u_h^{n+1}-u_h^{n},u_h^{n+1}-u_h^{n})-\|U^{n+1} - U_h^n\|^2,
\ea
\eq
independent of the size of $\Delta t$.
\end{thm}
\begin{proof} Taking $\phi=1$ in (\ref{ch_FPDGFull1st+}c) implies  (\ref{ch_ConMass}).
We next show the existence and uniqueness of (\ref{ch_FPDGFull1st+}) at each time step.
Substitution of (\ref{ch_FPDGFull1st+}b) into (\ref{ch_FPDGFull1st+}c)  with (\ref{ch_FPDGFull1st+}d) gives  the following linear system
\begin{subequations}\label{ch_FPDGFullAlg+}
\begin{align}
(u_h^{n+1}/\Delta t, \phi)+ A(M(u_h^n); w^{n+1}_h, \phi)= &(u_h^{n}/\Delta t, \phi),\\
A(\epsilon^2; u_h^{n+1}, \psi)+\left( 1/2H(u_h^n)^2u_h^{n+1}, \phi \right)-(w^{n+1}_h,\psi) = & \left( 1/2H(u_h^n)^2u_h^{n}, \phi \right)-(H(u_h^n)U_h^n, \psi).
\end{align}
\end{subequations}
It suffices to prove the uniqueness for this linear system.
Denoting $(\tilde u, \tilde w)$ the difference of two possible solutions of
(\ref{ch_FPDGFullAlg+}) for fixed $(u_h^n, w^n_h)$, so that
\begin{subequations} \label{FPDGFullAlg++}
\begin{align}
(\tilde u/\Delta t, \phi)+ A(M(u_h^n); \tilde w, \phi)= &0,\\
A(\epsilon^2; \tilde u, \psi)+\left( 1/2H(u_h^n)^2 \tilde u, \phi \right)-(\tilde w,\psi) = & 0.
\end{align}
\end{subequations}
Setting $\phi=\Delta t \tilde w, \psi=\tilde u$ and adding the two equations, we have
$$
\Delta t A(M(u_h^n); \tilde w, \tilde w)
  + A(\epsilon^2; \tilde u, \tilde u) + \left( 1/2H(u_h^n)^2\tilde u, \tilde u \right) =0.
$$
Using  (\ref{VCoer}) with $a=M(u_h^n)$ so that $\beta_0^*>\beta^*$, we have
\bqs
\ba
0 \geq & \Delta t M_{\min}  \|\tilde w\|_{DG}^2  + \epsilon^2 \|\tilde u\|_{DG}^2 + \frac{1}{2}\int_\Omega H(u_h^n)^2\tilde u^2 dx,
\ea
\eqs
which ensures that $\tilde u=const$ and $\tilde w=const$.
Then it follows $A(M(u_h^n); \tilde w, \phi)=A(\epsilon^2; \tilde u, \psi)=0$. Thus, (\ref{FPDGFullAlg++}a) is equivalent to
$$
(\tilde u, \phi) = 0, \quad \forall \phi \in V_h.
$$
We must have $\tilde u=0$.
In a similar fashion, $\tilde w=0$ follows from (\ref{FPDGFullAlg++}b). Hence the uniqueness for (\ref{ch_FPDGFullAlg+})
follows.


We next prove (\ref{ch_engdis1st+}). Taking $\phi=w_h^{n+1}$ in (\ref{ch_FPDGFull1st+}c),  $\psi=\frac{u_h^{n+1}-u_h^{n}}{\Delta t}$ in (\ref{ch_FPDGFull1st+}d) gives
\begin{align*}
- A(M(u_h^n); w_h^{n+1}, w_h^{n+1})  = A(\epsilon^2; u_h^{n+1}, \frac{u_h^{n+1}-u_h^{n}}{\Delta t})+\left( H(u_h^n)U^{n+1}, \frac{u_h^{n+1}-u_h^{n}}{\Delta t} \right).
\end{align*}
By (\ref{ch_FPDGFull1st+}b) and bilinearity of $A(\epsilon^2; \cdot, \cdot)$, the right hand side of the above equation gives
\begin{align*}
RHS = & \frac{1}{2\Delta t}\left( A(\epsilon^2; u_h^{n+1}, u_h^{n+1})- A(\epsilon^2; u_h^{n}, u_h^{n}) + A(\epsilon^2; u_h^{n+1}-u_h^{n}, u_h^{n+1}-u_h^{n}) \right) \\
&+\frac{1}{\Delta t} \left(\|U^{n+1}\|^2-\|U_h^n\|^2 + \| U^{n+1}-U_h^n\|^2 \right).
\end{align*}
Hence
\bq\label{ch_engdis1stO}
\begin{aligned}
E(u_h^{n+1}, U^{n+1}) = & E(u_h^n, U_h^n) - \Delta tA(M(u_h^n); w_h^{n+1},w_h^{n+1}) \\
&-\frac{1}{2}A(\epsilon^2; u_h^{n+1}-u_h^{n},u_h^{n+1}-u_h^{n})-\|U^{n+1} - U_h^n\|^2.
\end{aligned}
\eq
Implied by the fact that $\Pi$ is a contraction mapping in $L^2$,  we have
\bq\label{ch_engdis1stO1}
E(u_h^{n+1}, U_h^{n+1}) \leq E(u_h^{n+1}, U^{n+1}),
\eq
hence (\ref{ch_engdis1st+}) as desired.

\end{proof}

\begin{rem} Note that here $\beta_0>\beta_0^*$ serves only as a sufficient condition for stability.   For $M(u)=const$ and rectangular meshes,  $\beta_0^*$ can be more precisely estimated as $\beta_0^*=k^2$, see  \cite[Lemma 3.1]{L15}.  For general case, it suffices to add upon $k^2$ a factor of size $\max_K  \frac{\max_{x\in \partial K}  M(u_h^n)}{\min_{x\in K} M(u_h^n)}$, which approaches to one as the mesh is sufficiently refined. In our numerical examples we take $\beta_0^*= 3k^2$.
\end{rem}

\subsection{Second order fully discrete IEQ-DG scheme}
We first obtain $u_h^1, w_h^1$ and $U^1$ from the first order fully discrete IEQ-DG scheme (\ref{ch_FPDGFull1st+}).
We further use the second order backward differentiation formula (BDF2) for time discretization.  In other words, for $n\geq 1$,  
the second order fully discrete IEQ-DG scheme is to find $(u^{n+1}_h, w_h^{n+1}) \in V_h  \times V_h $ such that for $\forall \phi, \psi
\in V_h $,
\begin{subequations}\label{FPDGFull+}
	\begin{align}
	U_h^n= & \Pi U^n,\\
  \frac{3U^{n+1} - 4U_h^n + U_h^{n-1}}{2\Delta t}  =&  \frac{1}{2}H(u^{n,*}_h) \frac{3u_h^{n+1} - 4u_h^n+u_h^{n-1}}{2\Delta t},\\
	\left(  \frac{3u_h^{n+1} - 4u_h^n+u_h^{n-1}}{2\Delta t}, \phi \right)
	= & - A(M(u^{n, *}_h);w_h^{n+1}, \phi),\\
	(w_h^{n+1}, \psi) = & A(\epsilon^2;u_h^{n+1},\psi) + \left(H(u^{n,*}_h)U^{n+1},\psi \right),
	\end{align}
\end{subequations}
where $u^{n,*}_h$ is obtained using $u_h^{n-1}$ and $u^n_h$
\begin{align}\label{u8}
u^{n, *}_h=& 2u_h^n- u_h^{n-1}.
\end{align}
Here instead of  $u_h^{n+1}$ we use $u^{n, *}_h$ to avoid iteration steps in updating the numerical solution, while still maintaining second order accuracy in time.

To show the energy stability, we first present some useful identities.
\begin{lem}\label{bifor}
For any symmetric bilinear functional $\mathcal{A}(\cdot,\cdot)$, it follows
\begin{align*}
\mathcal{A}(\phi+\psi, \phi-\psi) & = \mathcal{A}(\phi, \phi)- \mathcal{A}(\psi, \psi),\\
2\mathcal{A}(\phi_1,3\phi_1-2\phi_2-\phi_3) & = \mathcal{A}(\phi_1,\phi_1)+\mathcal{A}(2\phi_1-\phi_2,2\phi_1-\phi_2)-\mathcal{A}(\phi_2,\phi_2)\\
& \qquad -\mathcal{A}(\phi_3,\phi_3)+\mathcal{A}(\phi_1-\phi_3,\phi_1-\phi_3).
\end{align*}
\end{lem}
\begin{proof} The first identity follows from a direct calculation using the symmetry of the bilinear functional $\mathcal{A}(\cdot,\cdot)$.
The second follows from a proper decomposition and using the first identity, that goes as follows:
\begin{equation*}
\begin{aligned}
2\mathcal{A}(\phi_1,3\phi_1-2\phi_2-\phi_3)= &\mathcal{A}(\phi_1,6\phi_1-4\phi_2-2\phi_3) \\
 =&  \mathcal{A}(\phi_1,\phi_1+4(\phi_1-\phi_2)+\phi_1-2\phi_3) \\
= & \mathcal{A}(\phi_1,\phi_1) + \mathcal{A}(2\phi_1,2\phi_1-2\phi_2)+\mathcal{A}(\phi_1,\phi_1-2\phi_3)\\
= & \mathcal{A}(\phi_1,\phi_1) + \mathcal{A}(2\phi_1-\phi_2,2\phi_1-\phi_2)-\mathcal{A}(\phi_2,\phi_2)\\
& +\mathcal{A}(\phi_1-\phi_3,\phi_1-\phi_3)-\mathcal{A}(\phi_3,\phi_3).
\end{aligned}
\end{equation*}
\end{proof}
For the scheme (\ref{FPDGFull+}), we have
\begin{thm}\label{thm2nd} Under the assumption in Theorem \ref{ch_firstorder+}, we set $\beta_0^*$ as defined in (\ref{beta0+}).
If $\beta_0>\beta_0^*$, the second order fully discrete DG scheme (\ref{FPDGFull+}) admits a unique solution $(u_h^{n+1}, w_h^{n+1})$, and the solution $u_h^n$ satisfies the mass conservation (\ref{ch_ConMass}) for any $n>0$.  Moreover, for any $\Delta t >0$ it follows
\bq\label{engdis}
\bar{E}^{n+1} - \bar{E}^n \leq - \Delta t A(M(u^{n, *}_h);w_h^{n+1},w_h^{n+1}) \leq 0,
\eq
where the modified energy is defined by
\bq\label{menergy}
\bar{E}^n=\frac{E(u_h^n,U_h^n)+E(u_h^{n,*},U_h^{n,*})}{2},
\eq
with
$$
U_h^{n,*}=2U_h^{n}-U_h^{n-1}.
$$
\end{thm}
\begin{proof}
We first prove (\ref{engdis}).
Taking $\phi=2\Delta t w_h^{n+1}$ in (\ref{FPDGFull+}c) 
gives
\bqs
\ba
- 2\Delta tA(M(u^{n, *}_h);w_h^{n+1}, w_h^{n+1})  = & (w_h^{n+1}, 3u_h^{n+1} - 4u_h^n+u_h^{n-1})
\\
\text{using} \; (\ref{FPDGFull+}d) \quad =& A(\epsilon^2;u_h^{n+1}, \psi) +
\left(H(u^{n,*}_h)U^{n+1}, \psi \right) \quad \psi:=3u_h^{n+1} - 4u_h^n+u_h^{n-1}\\
\text{using} \; (\ref{FPDGFull+}b) \quad =& A(\epsilon^2;u_h^{n+1}, 3u_h^{n+1} - 4u_h^n+u_h^{n-1}) +\left(3U^{n+1} - 4U_h^n + U_h^{n-1}, 2U^{n+1} \right)\\
= &  A(\epsilon^2;u_h^{n+1}, 3u_h^{n+1} - 2u_h^n - u_h^{n,*}) +\left(3U^{n+1} - 2U_h^n - U_h^{n,*}, 2U^{n+1} \right).
\ea
\eqs
Both $A(\epsilon^2; \cdot, \cdot)$ and $(\cdot, \cdot)$ are symmetric,  by Lemma \ref{bifor} we have
\bqs
\ba
 A(\epsilon^2;u_h^{n+1}, 3u_h^{n+1} - 2u_h^n - u_h^{n,*})= & \frac{1}{2} \left[ A(\epsilon^2;u_h^{n+1}, u_h^{n+1}) + A(\epsilon^2;u_h^{n+1,*}, u_h^{n+1,*} ) -A(\epsilon^2;u_h^{n}, u_h^{n}) \right. \\
& \left.- A(\epsilon^2;u_h^{n,*}, u_h^{n,*} )+ A(\epsilon^2;u_h^{n+1}-u_h^{n,*},u_h^{n+1}-u_h^{n,*}) \right],\\
 \left(3U^{n+1} - 2U_h^n - U_h^{n,*}, 2U^{n+1} \right)=& \|U^{n+1}\|^2+\|2U^{n+1}-U_h^{n}\|^2-\|U_h^{n}\|^2-\|U_h^{n,*}\|^2 + \|U^{n+1}-U_h^{n,*}\|^2.
\ea
\eqs
Regrouping, we obtain
\begin{equation}\label{engdis2nd+}
\ba
& \frac{1}{2} \left[ A(\epsilon^2;u_h^{n+1}, u_h^{n+1}) + A(\epsilon^2;u_h^{n+1,*}, u_h^{n+1,*} ) \right]+\|U^{n+1}\|^2+\|2U^{n+1}-U_h^{n}\|^2 \\
& =  2\bar{E}^n - 2\Delta t A(M(u^{n, *}_h);w_h^{n+1},w_h^{n+1}) -\frac{1}{2} A(\epsilon^2;u_h^{n+1}-u_h^{n,*},u_h^{n+1}-u_h^{n,*})-\|U^{n+1}-U_h^{n,*}\|^2\\
& \leq 2\bar{E}^n - 2\Delta t A(M(u^{n, *}_h);w_h^{n+1},w_h^{n+1}).
\ea
\end{equation}
Further use of the fact that $\Pi$ is a contraction mapping in $L^2$, we have
$$
\|U_h^{n+1}\|^2\leq\|U^{n+1}\|^2, \quad \|2U_h^{n+1}-U_h^{n}\|^2\leq \|2U^{n+1}-U_h^{n}\|^2.
$$
Then the left hand side of (\ref{engdis2nd+}) is bounded below by $2\bar{E}^{n+1}$, thus (\ref{engdis}) follows.

Similar to the proof of Theorem \ref{ch_firstorder+}, the existence and uniqueness of the scheme (\ref{FPDGFull+}) is equivalent to showing the uniqueness of $u_h^{n+1}, w_h^{n+1}$ given $u_h^i, w_h^i, U^i$ with $i=n, n-1$.

Let $(\tilde u, \tilde w, \tilde U)$ be the difference of two such solutions, then
\begin{subequations}\label{FPDGFull+diff}
	\begin{align}
  \tilde U   =&  \frac{1}{2}H(u^{n,*}_h) \tilde u,\\
	\left(3\tilde u, \phi \right)
	= & - 2\Delta t A(M(u^{n, *}_h);\tilde w, \phi),\\
	(\tilde w, \psi) = & A(\epsilon^2;\tilde u,\psi) + \left(H(u^{n,*}_h)\tilde U,\psi \right).
	\end{align}
\end{subequations}
Setting $\phi= \tilde w, \ \psi=3\tilde u$,  and subtracting (\ref{FPDGFull+diff}b) from (\ref{FPDGFull+diff}c), it follows
$$
2\Delta tA(M(u^{n, *}_h);\tilde w, \tilde w) + 3A(\epsilon^2;\tilde u, \tilde u)+6\|\tilde U\|^2=0,
$$
where (\ref{FPDGFull+diff}a) has been used to simplify the third term.  By (\ref{VCoer}) with $a=M(u_h^n)$, we have
$\beta_0>\beta_0^*>\beta^*$, hence
$$
2 \Delta t M_{\min}  \|\tilde w\|_{DG}^2 +3\epsilon^2 \|\tilde u\|^2_{DG}+6\|\tilde U\|^2\leq 0,
$$
which ensures that $\tilde u=const$, $\tilde w=const$ and $\tilde U=0$.
Thus, using (\ref{FPDGFull+diff}) again, we must have $\tilde u=\tilde w=0$. Thus leads to the existence and uniqueness of the scheme (\ref{FPDGFull+}).

Taking $\phi=1$ in (\ref{FPDGFull+}c),
it follows
\bq\label{ch_ConMass2nd}
\int_\Omega u_h^{n+1} dx = \frac{1}{3}\int_\Omega  4u_h^n-u_h^{n-1} dx.
\eq
From Theorem \ref{ch_firstorder+}, we have
\bq\label{ch_ConMass1}
\int_\Omega u_h^{1} dx = \int_\Omega u_h^{0} dx,
\eq
which when combined with (\ref{ch_ConMass2nd}) gives the mass conservation (\ref{ch_ConMass}).
\end{proof}

\subsection{Algorithms}
%

The details related to the schemes implementation are summarized in the following algorithms.

\subsubsection{Algorithm for the first order fully discrete IEQ-DG scheme (\ref{ch_FPDGFull1st+})}\label{algorithm1st}
\begin{itemize}
  \item Step 1 (Initialization) From the given initial data $u_0(x)$
  \begin{enumerate}
    \item Generate $u_h^0 =\Pi u_0(x) \in V_h $; 
    \item Generate $U^0= \sqrt{F(u_0(x)) +B}$, where $B$ is a priori chosen so that $\inf (F(v)+B)>0$.
  \end{enumerate}
  \item Step 2 (Evolution)
  \begin{enumerate}
  \item Project $U^n$ back into $V_h$, namely compute $U^n_h=\Pi U^n$;
    \item Solve the linear system (\ref{ch_FPDGFullAlg+}) for $u_h^{n+1}, w_h^{n+1}$;
      \item Update $U^{n+1}$ using (\ref{ch_FPDGFull1st+}b), then return to (1) in Step 2.
  \end{enumerate}
\end{itemize}

\subsubsection{Algorithm for the second order fully discrete IEQ-DG scheme (\ref{FPDGFull+})}
\begin{itemize}
  \item Step 1 (Initialization) From the given initial data $u_0(x)$
  \begin{enumerate}
    \item Generate $u_h^0 =\Pi u_0(x) \in V_h $; 
    \item Generate $U^0= \sqrt{F(u_0(x)) +B}$, where $B$ is a priori chosen so that $\inf (F(v)+B)>0$; and
    \item Solve for $u_h^1, w_h^1$ and $U^1$ for $\forall x \in S$ through Algorithm \ref{algorithm1st} for the first order fully discrete IEQ-DG scheme (\ref{ch_FPDGFull1st+}).
  \end{enumerate}
  \item Step 2 (Evolution)
  \begin{enumerate}
  \item Project $U^n$ back into $V_h$, namely compute $U^n_h=\Pi U^n$;
    \item Solve the linear system for $u_h^{n+1}, w_h^{n+1}$,
    \bqs
    \ba
    \left(  \frac{3u_h^{n+1}}{2\Delta t}, \phi \right)+ A(M(u^{n, *}_h);w_h^{n+1}, \phi)
	= & \left(  \frac{4u_h^n-u_h^{n-1}}{2\Delta t}, \phi \right),\\
	  A(\epsilon^2;u_h^{n+1},\psi) + \frac{1}{2}\left(\left(H(u^{n,*}_h)\right)^2u_h^{n+1},\psi \right)-(w_h^{n+1}, \psi) =& RHS,
    \ea
    \eqs
    where $RHS=-\left( H(u^{n,*}_h)\frac{4U_h^n-U_h^{n-1}}{3}- \frac{1}{2} \left(H(u^{n,*}_h)\right)^2\frac{4u_h^n-u_h^{n-1}}{3} , \psi \right)$
      \item Update $U^{n+1}$ through (\ref{FPDGFull+}b), i.e.,
      $$
      U^{n+1}=\frac{1}{2}H(u^{n,*}_h)u_h^{n+1}+\left( \frac{4U_h^n-U_h^{n-1}}{3}- \frac{1}{2}H(u^{n,*}_h)\frac{4u_h^n-u_h^{n-1}}{3} \right),
      $$
      then return to (1) in Step 2.
  \end{enumerate}
\end{itemize}
\begin{rem} Higher order (in time) IEQ discretization is possible. We omit the details here due to space limitation.
Interested readers are referred to  \cite{GZW19}  for some arbitrarily high-order linear schemes for gradient flow models.
\end{rem}

\section{Extensions}\label{sec4}
It is known that solving the Cahn-Hilliard equation with degenerate mobility and/or logarithmic potential is more difficult since
it requires a point-wise control of the numerical solution.    We discuss how our schemes can be applied
by a proper modification.


\subsection{Mobility} Though the mobility is often taken as a constant for simplicity,  a thermodynamically reasonable choice is actually the degenerate mobility $M(u)=u(1-u)$ (see e.g., \cite{EG96}).   There is hope that solutions which initially take values in the interval $[0, 1]$ will do so for all positive time (which is not true for fourth-order parabolic equations without degeneracy). We remark that only values in the interval $[0, 1]$ are physically meaningful. Such degeneracy leads to numerical difficulties.

Here, we follow \cite{EG96, BBG99} by considering the modified mobility
\bq\label{mobility}
\widetilde{M}(u) =
\left \{
\begin{array}{rl}
	M(\sigma) \quad & u\leq\sigma, \\
	M(u)  \quad &\sigma < u < 1-\sigma, \\
	M(1-\sigma)  \quad & u\geq1-\sigma,\\
\end{array}
\right.
\eq
It is obvious that for given $\sigma$,
$$
\widetilde{M}(u) \geq M_{\min} >0,
$$
and it is well-defined for $u \in (-\infty, \infty)$. Numerically, we apply our scheme using this modified mobility with a small $\sigma$.



\subsection{Flory-Huggins potential}
A practical choice for the potential is the Logarithmic Flory-Huggins function \cite{BE91, CH58, CE92}
\bq\label{fhp}
F(v) = \frac{\theta}{2}\left( v \ln v +(1-v) \ln (1-v)\right) + \frac{\theta_c}{2} v(1-v), \quad  \ v \in [0,1],
\eq
where $\theta, \theta_c>0$ are physical parameters.  This function is non-convex with double wells for $\theta_c>2\theta$,  and it only has a single well and admits only a single phase  for $\theta_c\leq 2\theta$ \cite{WKG06}.

The domain of the logarithmic potential (\ref{fhp}) is  $(0,1)$, which requires the numerical solution
be strictly inside $(0,1)$.  For some numerical schemes, such solution bounds can be established (see, e.g.,  \cite{DD95, EG96, MZ04, CE92, CWWW19}).

For high order DG schemes it is rather difficult to preserve the numerical solution within $(0, 1)$. We choose to  regularize the logarithmic Flory-Huggins potential (\ref{fhp}) by extending its domain from $(0,1)$ to $(-\infty, \infty)$.
Such regularization technique is commonly used to remove the numerical overflow; see, e.g., \cite{CE92, BB95, EG96, BB99, YZ19}. Specifically, it can be  replaced by the twice continuously differentiable function
\bqs\label{rfhp}
\widetilde{F}(v) =
\left \{
\begin{array}{rl}
	\frac{\theta}{2} \left(v \ln v +(1-v) \ln \sigma + \frac{(1-v)^2}{2\sigma} - \frac{\sigma}{2} \right)+ \frac{\theta_c}{2} v(1-v),  \quad & v \geq 1-\sigma, \\
	\frac{\theta}{2}\left( v \ln v +(1-v) \ln (1-v)\right) + \frac{\theta_c}{2} v(1-v), \quad &  \sigma < v < 1- \sigma,\\
	\frac{\theta}{2} \left((1-v) \ln (1-v) +v\ln \sigma + \frac{v^2}{2\sigma} - \frac{\sigma}{2} \right)+ \frac{\theta_c}{2} v(1-v),  \quad & v \leq \sigma,\\
\end{array}
\right.
\eqs
and thus $\widetilde{F}(v)$ is well defined for $v \in (-\infty, \infty)$.
It was argued in \cite{EG96} that the solution with regularized $\widetilde{M}(u)$ and $\widetilde{F}(u)$ converges to the solution to the original problem as $\sigma \to 0$. 
This treatment has been applied in numerical simulations, for example in \cite{BBG99}. In this paper, we apply our IEQ-DG schemes to problems formulated with the modified mobility and the regularized potential.

\section{Numerical examples}\label{sec5}
In this section, we will carry out several numerical tests in both 1D and 2D to demonstrate both temporal and spatial accuracy of the IEQ-DG schemes (\ref{ch_FPDGFull1st+}) and (\ref{FPDGFull+}), the mass conservation and energy dissipation properties. For the spatial accuracy, we will choose $\Delta t$ sufficiently small such that the spatial discretization error is dominant. Likewise, for the temporal accuracy, we will set spatial meshes sufficiently refined such that temporal discretization error is dominant. In the following numerical examples, the parameter $\beta_0=k^2+0.5k$ for problems with constant mobility and $\beta_0=3k^2+0.5k$ for other cases. The parameter $B=1$ as default unless specified.

\begin{example}\label{exam1dp} (1D spatial accuracy test)
Consider the Cahn-Hilliard equation (\ref{CH}) with $M=1$ and double-well potential $F(u)=\frac{1}{4}(u^2-1)^2$ in $\Omega = [0, 2\pi]$ with periodic boundary conditions.
Here, we follow Example 5.2 in \cite{SS17} by adding a source term
\begin{equation}\label{1dsource}
s(x,t) = -e^{-t}\sin x \left(3e^{-2t}\cos 2x  + 3e^{-2t}\cos^2 x + 1\right)
\end{equation}
to the Cahn-Hilliard equation (\ref{CH}), so that the exact solution is
\begin{equation}\label{1dexact}
u(x,t) = e^{-t}\sin x.
\end{equation}
We use the fully discrete IEQ-DG scheme (\ref{FPDGFull+}) with a term
$\left( s(x, t^{n+1}), \phi \right)$ added to the right hand side of  (\ref{FPDGFull+}c), and we test the DG scheme based on $P^k$ polynomials, with $k=1,2,3$. Both errors and orders of accuracy at $T=1$ are reported in Table \ref{tab1dp}. These results show that $(k+1)$th order of accuracy in both $L^2$ and $L^\infty$ norms are obtained.

\begin{table}[!htbp]\tabcolsep0.03in
\centering
\caption{1D $L^2, \; L^\infty$ errors and orders of accuracy at $T= 1$.}
\begin{tabular}[c]{||c|c|c|c|c|c|c|c|c|c||}
\hline
\multirow{2}{*}{$k$} & \multirow{2}{*}{$\Delta t$}&   \multirow{2}{*}{ } & N=10 & \multicolumn{2}{|c|}{N=20} & \multicolumn{2}{|c|}{N=40} & \multicolumn{2}{|c||}{N=80}  \\
\cline{4-10}
& & & error & error & order & error & order & error & order\\
\hline
\multirow{2}{*}{1}  & \multirow{2}{*}{1e-3} & $\|u-u_h\|_{L^2}$ &  3.09646e-02 & 8.07876e-03 & 1.94 & 2.03575e-03 & 1.99 & 5.10124e-04 & 2.00  \\
\cline{3-10}
 & & $\|u-u_h\|_{L^\infty}$  & 1.68270e-02 & 4.58886e-03 & 1.87 & 1.16103e-03	
 & 1.98 & 2.91198e-04 & 2.00  \\
\hline
\multirow{2}{*}{2}  & \multirow{2}{*}{1e-4} & $\|u-u_h\|_{L^2}$ & 3.56585e-04  & 4.17179e-05 & 3.10 & 5.12149e-06 & 3.03 & 6.35139e-07 & 3.01  \\
\cline{3-10}
 & & $\|u-u_h\|_{L^\infty}$  & 4.34261e-04 & 5.50274e-05 & 2.98 & 6.89646e-06 & 3.00 & 8.63616e-07 & 3.00  \\
\hline
\multirow{2}{*}{3} & \multirow{2}{*}{1e-5}  & $\|u-u_h\|_{L^2}$ & 2.57355e-05 & 1.66983e-06 & 3.95 & 1.05343e-07 & 3.99 & 6.62827e-09 & 3.99  \\
\cline{3-10}
 & & $\|u-u_h\|_{L^\infty}$   & 1.86828e-05 & 1.27898e-06 & 3.87 & 8.12040e-08 & 3.98 & 5.31052e-09 & 3.93  \\
\hline
\end{tabular}\label{tab1dp}
\end{table}

Note that the test error in Table \ref{tab1dp} is expected to be of order $O(\Delta t^2 + h^{k+1})$.  In order to observe the desired order in space, time step needs to be smaller so that $\Delta t \leq O(h^{k+1}/2)$ when $k$ gets larger. This comment applies to other cases as well.

\end{example}

\begin{example} (2D spatial accuracy test with constant mobility and double-well potential) For the Cahn-Hilliard equation (\ref{CH}) with $M(u)=1$ and double-well potential $F(u)=\frac{1}{4}(u^2-1)^2$ in $\Omega$ with appropriate boundary conditions, we add a source term
$$
s(x,y,t)=-\frac{w(x,y,t)}{4}+\frac{\epsilon^2 w(x,y,t)}{4} - \frac{3w(x,y,t)v(x,y,t) }{2}+\frac{3w(x,y,t)^3}{2}-\frac{w(x,y,t)}{2}
$$
to the right hand side of (\ref{CH}), where
\bqs
\ba
w(x,y,t)=&0.1e^{-t/4}\sin(x/2)\sin(y/2),\\
v(x,y,t)=&\left(0.1e^{-t/4}\cos(x/2)\sin(y/2)\right)^2+\left(0.1e^{-t/4}\sin(x/2)\cos(y/2)\right)^2,\\
\ea
\eqs
so that the exact solution is
$$
u(x,y,t) = w(x,y,t).
$$
Here the parameter $\epsilon=0.1$.
We test this example by DG scheme (\ref{ch_FPDGFull1st+}) with a term
$\left( s(x,y, t^{n+1}), \phi \right)$  added to the right hand side of (\ref{ch_FPDGFull1st+}c), and the DG scheme is based on polynomials of degree $k$ with $k=1, 2, 3$ on rectangular meshes.

\noindent\textbf{Test case 1.} (Periodic BC) In this test case, we take $\Omega=[0, 4\pi]^2$ and consider periodic boundary conditions.
Both errors and orders of accuracy at $T=0.01$ are reported in Table \ref{ch_tab2dtk1}. These results show that $(k+1)$th order of accuracy in both $L^2$ and $L^{\infty}$ are obtained.
\begin{table}[!htbp]\tabcolsep0.03in
\centering
\caption{2D $L^2, \; L^\infty$ errors at $T= 0.01$ with mesh $N\times N$.}
\begin{tabular}[c]{||c|c|c|c|c|c|c|c|c|c||}
\hline
\multirow{2}{*}{$k$} & \multirow{2}{*}{$\Delta t$}&   \multirow{2}{*}{ } & N=8 & \multicolumn{2}{|c|}{N=16} & \multicolumn{2}{|c|}{N=32} & \multicolumn{2}{|c||}{N=64}  \\
\cline{4-10}
& & & error & error & order & error & order & error & order\\
\hline
\multirow{2}{*}{1}  & \multirow{2}{*}{1e-3} & $\|u-u_h\|_{L^2}$ &  3.16822e-02 & 8.03463e-03 & 1.98 & 2.02336e-03 & 1.99 & 5.04024e-04 & 2.01  \\
\cline{3-10}
 & & $\|u-u_h\|_{L^\infty}$  & 1.38669e-02 & 3.74776e-03 & 1.89 & 9.59555e-04 & 1.97 & 2.40239e-04 & 2.00  \\
\hline
\hline
\multirow{2}{*}{2}  & \multirow{2}{*}{1e-4} & $\|u-u_h\|_{L^2}$ & 4.52729e-03 & 5.75115e-04 & 2.98 & 7.33589e-05 & 2.97 & 9.21578e-06 & 2.99  \\
\cline{3-10}
 & & $\|u-u_h\|_{L^\infty}$  & 2.32640e-03 & 2.95229e-04 & 2.98 & 4.06866e-05 & 2.86 & 5.26926e-06 & 2.95  \\
\hline
\hline
\multirow{2}{*}{3} & \multirow{2}{*}{1e-5}  & $\|u-u_h\|_{L^2}$ & 4.46670e-04 & 2.97916e-05 & 3.91 & 1.89117e-06 & 3.98 & 1.18585e-07 & 4.00  \\
\cline{3-10}
 & & $\|u-u_h\|_{L^\infty}$   & 3.20555e-04 & 1.80104e-05 & 4.15 & 1.02204e-06 & 4.14 & 6.16224e-08 & 4.05  \\
\hline
\end{tabular}\label{ch_tab2dtk1}
\end{table}

\noindent\textbf{Test case 2.} (Neumann BC) Considering $\Omega=[-\pi, 3\pi]^2$ with homogenous Neumann boundary conditions (\ref{BC}(ii)), both errors and orders of accuracy at $T=0.01$ are reported in Table \ref{tab2dtk2}. These results also show $(k+1)$th order of accuracy in both $L^2$ and $L^{\infty}$.

\begin{table}[!htbp]\tabcolsep0.03in
\centering
\caption{2D $L^2, \; L^\infty$ errors at $T= 0.01$ with mesh $N\times N$.}
\begin{tabular}[c]{||c|c|c|c|c|c|c|c|c|c||}
\hline
\multirow{2}{*}{$k$} & \multirow{2}{*}{$\Delta t$}&   \multirow{2}{*}{ } & N=8 & \multicolumn{2}{|c|}{N=16} & \multicolumn{2}{|c|}{N=32} & \multicolumn{2}{|c||}{N=64}  \\
\cline{4-10}
& & & error & error & order & error & order & error & order\\
\hline
\multirow{2}{*}{1}  & \multirow{2}{*}{1e-3} & $\|u-u_h\|_{L^2}$ &  3.16822e-02 & 8.03463e-03 & 1.98 & 2.02336e-03 & 1.99 & 5.04024e-04 & 2.01  \\
\cline{3-10}
 & & $\|u-u_h\|_{L^\infty}$  & 1.38669e-02 & 3.74776e-03 & 1.89 & 9.59555e-04 & 1.97 & 2.40239e-04 & 2.00  \\
\hline
\hline
\multirow{2}{*}{2}  & \multirow{2}{*}{1e-4} & $\|u-u_h\|_{L^2}$ &  4.52729e-03 & 5.75115e-04 & 2.98 & 7.33591e-05 & 2.97 & 9.18427e-06 & 3.00  \\
\cline{3-10}
 & & $\|u-u_h\|_{L^\infty}$  & 2.32640e-03 & 2.95229e-04 & 2.98 & 4.06885e-05 & 2.86 & 5.08342e-06 & 3.00  \\
\hline
\hline
\multirow{2}{*}{3} & \multirow{2}{*}{1e-5}  & $\|u-u_h\|_{L^2}$ &  4.46670e-04 & 2.97916e-05 & 3.91 & 1.89102e-06 & 3.98 & 1.18133e-07 & 4.00  \\
\cline{3-10}
 & & $\|u-u_h\|_{L^\infty}$   & 3.20555e-04 & 1.80104e-05 & 4.15 & 1.02406e-06 & 4.14 & 6.40520e-08 & 4.00  \\
\hline
\end{tabular}\label{tab2dtk2}
\end{table}

\end{example}

\begin{example}\label{exam2dcl} (2D spatial accuracy test with constant mobility and logarithmic potential) We consider the Cahn-Hilliard equation (\ref{CH}) with constant mobility $M(u)=1$, the logarithmic Flory-Huggins potential (\ref{fhp}) with $\theta=\theta_c=2$, the parameters $\epsilon=1$ and $B=10$. We add an appropriate source term $s(x,y,t)$ to the right hand side of (\ref{CH}) such that the exact solution is
$$
u(x,y,t)=\frac{1}{10}e^{-t/4}\sin(x/4)\sin(y/4)+\frac{1}{2}.
$$

We test this example by DG scheme (\ref{FPDGFull+}) with a term
$\left( s(x,y, t^{n+1}), \phi \right)$  added to the right hand side of (\ref{FPDGFull+}c), and the DG scheme is also based on polynomials of degree $k$ with $k=1, 2, 3$ on rectangular meshes.

\noindent\textbf{Test case 1.} (Periodic BC) In this test case, we take $\Omega=[0, 8\pi]^2$ and consider periodic boundary conditions.
Both errors and orders of accuracy at $T=0.01$ are reported in Table \ref{tab2dcln1}. These results show that $(k+1)$th order of accuracy in both $L^2$ and $L^{\infty}$ are obtained.

\begin{table}[!htbp]\tabcolsep0.03in
\centering
\caption{2D $L^2, \; L^\infty$ errors at $T= 0.01$ with mesh $N\times N$.}
\begin{tabular}[c]{||c|c|c|c|c|c|c|c|c|c||}
\hline
\multirow{2}{*}{$k$} & \multirow{2}{*}{$\Delta t$}&   \multirow{2}{*}{ } & N=8 & \multicolumn{2}{|c|}{N=16} & \multicolumn{2}{|c|}{N=32} & \multicolumn{2}{|c||}{N=64}  \\
\cline{4-10}
& & & error & error & order & error & order & error & order\\
\hline
\multirow{2}{*}{1}  & \multirow{2}{*}{1e-3} & $\|u-u_h\|_{L^2}$ &  6.34010e-02 & 1.62047e-02 & 1.97 & 4.04183e-03 & 2.00 & 1.00777e-03 & 2.00  \\
\cline{3-10}
 & & $\|u-u_h\|_{L^\infty}$  &  1.38744e-02 & 3.74858e-03 & 1.89 & 9.55245e-04 & 1.97 & 2.39967e-04 & 1.99  \\
\hline
\hline
\multirow{2}{*}{2}  & \multirow{2}{*}{1e-4} & $\|u-u_h\|_{L^2}$ & 9.39224e-03 & 1.18059e-03 & 2.99 & 1.46853e-04 & 3.01 & 1.83323e-05 & 3.00  \\
\cline{3-10}
 & & $\|u-u_h\|_{L^\infty}$  & 2.45698e-03 & 3.14143e-04 & 2.97 & 3.74571e-05 & 3.07 & 4.54860e-06 & 3.04  \\
\hline
\hline
\multirow{2}{*}{3} & \multirow{2}{*}{5e-6}  & $\|u-u_h\|_{L^2}$ &  1.09183e-03 & 6.72768e-05 & 4.02 & 4.09870e-06 & 4.04 & 2.54225e-07 & 4.01  \\
\cline{3-10}
 & & $\|u-u_h\|_{L^\infty}$   &  2.30167e-04 & 1.58541e-05 & 3.86 & 1.02039e-06 & 3.96 & 6.42180e-08 & 3.99  \\
\hline
\end{tabular}\label{tab2dcln1}
\end{table}

\noindent\textbf{Test case 2.} (Neumann BC) In this test case, we take $\Omega=[-2\pi, 2\pi]^2$ and consider Neumann boundary conditions.
Both errors and orders of accuracy at $T=0.01$ are reported in Table \ref{tab2dcln2}. These results show that $(k+1)$th order of accuracy in both $L^2$ and $L^{\infty}$ are obtained.

\begin{table}[!htbp]\tabcolsep0.03in
\centering
\caption{2D $L^2, \; L^\infty$ errors at $T= 0.01$ with mesh $N\times N$.}
\begin{tabular}[c]{||c|c|c|c|c|c|c|c|c|c||}
\hline
\multirow{2}{*}{$k$} & \multirow{2}{*}{$\Delta t$}&   \multirow{2}{*}{ } & N=8 & \multicolumn{2}{|c|}{N=16} & \multicolumn{2}{|c|}{N=32} & \multicolumn{2}{|c||}{N=64}  \\
\cline{4-10}
& & & error & error & order & error & order & error & order\\
\hline
\multirow{2}{*}{1}  & \multirow{2}{*}{1e-3} & $\|u-u_h\|_{L^2}$ &  1.27997e-01 & 3.55296e-02 & 1.85 & 9.55174e-03 & 1.90 & 2.13203e-03 & 2.16  \\
\cline{3-10}
 & & $\|u-u_h\|_{L^\infty}$  & 5.54685e-02 & 1.49970e-02 & 1.89 & 3.92808e-03 & 1.93 & 9.67492e-04 & 2.02  \\
\hline
\hline
\multirow{2}{*}{2}  & \multirow{2}{*}{1e-4} & $\|u-u_h\|_{L^2}$ & 1.87014e-02 & 2.35480e-03 & 2.99 & 2.94393e-04 & 3.00 & 3.69614e-05 & 2.99  \\
\cline{3-10}
 & & $\|u-u_h\|_{L^\infty}$  &  1.05130e-02 & 1.30587e-03 & 3.01 & 1.62919e-04 & 3.00 & 2.04032e-05 & 3.00  \\
\hline
\hline
\multirow{2}{*}{3} & \multirow{2}{*}{5e-6}  &  $\|u-u_h\|_{L^2}$ & 2.23974e-03 & 1.24902e-04 & 4.16 & 7.57937e-06 & 4.04 & 4.99051e-07 & 3.92  \\
\cline{3-10}
 & & $\|u-u_h\|_{L^\infty}$   & 1.47731e-03 & 8.22721e-05 & 4.17 & 4.36207e-06 & 4.24 & 3.29965e-07 & 3.72  \\
\hline
\end{tabular}\label{tab2dcln2}
\end{table}

\end{example}

\begin{example} (2D spatial accuracy test with degenerate mobility and logarithmic potential) We consider the Cahn-Hilliard equation (\ref{CH}) with degenerate mobility $M(u)=u(1-u)$, the logarithmic Flory-Huggins potential (\ref{fhp}) with $\theta=\theta_c=2$, the parameters $\epsilon=1$ and $B=10$. We add an appropriate source term $s(x,y,t)$ to the right hand side of (\ref{CH}) such that the exact solution is
$$
u(x,y,t)=\frac{2}{5}e^{-t/4}\sin(x/2)\sin(y/2)+\frac{1}{2}.
$$

We test this example by DG scheme (\ref{ch_FPDGFull1st+}) with a term
$\left( s(x,y, t^{n+1}), \phi \right)$  added to the right hand side of (\ref{ch_FPDGFull1st+}c), and the DG scheme is also based on polynomials of degree $k$ with $k=1, 2, 3$ on rectangular meshes.

\noindent\textbf{Test case 1.} (Periodic BC) In this test case, we take $\Omega=[0, 4\pi]^2$ and consider periodic boundary conditions.
Both errors and orders of accuracy at $T=0.01$ are reported in Table \ref{tab2ddln1}. These results show that $(k+1)$th order of accuracy in both $L^2$ and $L^{\infty}$ are obtained.

\begin{table}[!htbp]\tabcolsep0.03in
\centering
\caption{2D $L^2, \; L^\infty$ errors at $T= 0.01$ with mesh $N\times N$.}
\begin{tabular}[c]{||c|c|c|c|c|c|c|c|c|c||}
\hline
\multirow{2}{*}{$k$} & \multirow{2}{*}{$\Delta t$}&   \multirow{2}{*}{ } & N=8 & \multicolumn{2}{|c|}{N=16} & \multicolumn{2}{|c|}{N=32} & \multicolumn{2}{|c||}{N=64}  \\
\cline{4-10}
& & & error & error & order & error & order & error & order\\
\hline
\multirow{2}{*}{1}  & \multirow{2}{*}{1e-3} & $\|u-u_h\|_{L^2}$ &  1.31235e-01 & 3.29574e-02 & 1.99 & 8.27934e-03 & 1.99 & 2.08160e-03 & 1.99 \\
\cline{3-10}
 & & $\|u-u_h\|_{L^\infty}$  & 5.56010e-02 & 1.49372e-02 & 1.90 & 3.81584e-03 & 1.97 & 9.59510e-04 & 1.99  \\
\hline
\hline
\multirow{2}{*}{2}  & \multirow{2}{*}{1e-4} & $\|u-u_h\|_{L^2}$ & 2.05688e-02 & 2.51806e-03 & 3.03 & 3.05650e-04 & 3.04 & 3.79714e-05 & 3.01  \\
\cline{3-10}
 & & $\|u-u_h\|_{L^\infty}$  & 1.13806e-02 & 1.32194e-03 & 3.11 & 1.48147e-04 & 3.16 & 1.77820e-05 & 3.06  \\
\hline
\hline
\multirow{2}{*}{3} & \multirow{2}{*}{5e-6}  & $\|u-u_h\|_{L^2}$ &  2.82305e-03 & 1.48385e-04 & 4.25 & 8.56909e-06 & 4.11 & 5.53886e-07 & 3.95  \\
\cline{3-10}
 & & $\|u-u_h\|_{L^\infty}$   & 1.58906e-03 & 9.24779e-05 & 4.10 & 4.63277e-06 & 4.32 & 3.35743e-07 & 3.79  \\
\hline
\end{tabular}\label{tab2ddln1}
\end{table}

\noindent\textbf{Test case 2.} (Neumann BC) In this test case, we take $\Omega=[-\pi, 3\pi]^2$ and consider Neumann boundary conditions.
Both errors and orders of accuracy at $T=0.01$ are reported in Table \ref{tab2ddln2}. These results show that $(k+1)$th order of accuracy in both $L^2$ and $L^{\infty}$ is obtained.

\begin{table}[!htbp]\tabcolsep0.03in
\centering
\caption{2D $L^2, \; L^\infty$ errors at $T= 0.01$ with mesh $N\times N$.}
\begin{tabular}[c]{||c|c|c|c|c|c|c|c|c|c||}
\hline
\multirow{2}{*}{$k$} & \multirow{2}{*}{$\Delta t$}&   \multirow{2}{*}{ } & N=8 & \multicolumn{2}{|c|}{N=16} & \multicolumn{2}{|c|}{N=32} & \multicolumn{2}{|c||}{N=64}  \\
\cline{4-10}
& & & error & error & order & error & order & error & order\\
\hline
\multirow{2}{*}{1}  & \multirow{2}{*}{1e-3} & $\|u-u_h\|_{L^2}$ &  1.31235e-01 & 3.29574e-02 & 1.99 & 8.27934e-03 & 1.99 & 2.08160e-03 & 1.99  \\
\cline{3-10}
 & & $\|u-u_h\|_{L^\infty}$  & 5.56010e-02 & 1.49372e-02 & 1.90 & 3.81584e-03 & 1.97 & 9.59510e-04 & 1.99  \\
\hline
\hline
\multirow{2}{*}{2}  & \multirow{2}{*}{1e-4} & $\|u-u_h\|_{L^2}$ & 2.05688e-02 & 2.51806e-03 & 3.03 & 3.05650e-04 & 3.04 & 3.79715e-05 & 3.01  \\
\cline{3-10}
 & & $\|u-u_h\|_{L^\infty}$  & 1.13806e-02 & 1.32194e-03 & 3.11 & 1.48147e-04 & 3.16 & 1.77820e-05 & 3.06  \\
\hline
\hline
\multirow{2}{*}{3} & \multirow{2}{*}{5e-6}  &  $\|u-u_h\|_{L^2}$ &  2.82305e-03 & 1.48385e-04 & 4.25 & 8.56909e-06 & 4.11 & 5.59243e-07 & 3.94  \\
\cline{3-10}
 & & $\|u-u_h\|_{L^\infty}$   & 1.58906e-03 & 9.24779e-05 & 4.10 & 4.63278e-06 & 4.32 & 3.42344e-07 & 3.76  \\
\hline
\end{tabular}\label{tab2ddln2}
\end{table}

\end{example}

\begin{example} (Temporal Accuracy Test)
Following the test case 2 in Example \ref{exam2dcl}, we produce numerical solutions at $T=1$ using DG schemes (\ref{ch_FPDGFull1st+}) and (\ref{FPDGFull+}) based on $P^2$ polynomails with time steps $\Delta t= 2^{-m}$ with $2\leq m \leq 5$ and appropriate meshes. The $L^2, L^\infty$ errors and orders of convergence are shown in Table \ref{timeacc}, and these results confirm that DG schemes (\ref{ch_FPDGFull1st+}) and (\ref{FPDGFull+}) are first order and second order in time, respectively.

\begin{table}[!htbp]\tabcolsep0.03in
\caption{$L^2, L^\infty$ errors and EOC at $T = 1$ with time step $\Delta t$.}
\begin{tabular}[c]{||c|c|c|c|c|c|c|c|c|c||}
\hline
\multirow{2}{*}{Scheme} & \multirow{2}{*}{Mesh}&   \multirow{2}{*}{ } & $\Delta t=2^{-2}$ & \multicolumn{2}{|c|}{$\Delta t=2^{-3}$} & \multicolumn{2}{|c|}{$\Delta t=2^{-4}$} & \multicolumn{2}{|c||}{$\Delta t=2^{-5}$}  \\
\cline{4-10}
& & & error & error & order & error & order & error & order\\
\hline
\multirow{2}{*}{(\ref{ch_FPDGFull1st+})}  & \multirow{2}{*}{$32^2$} & $\|u-u_h\|_{L^2}$ &  4.21032e-03 & 2.06620e-03 & 1.03 & 1.02380e-03 & 1.01 & 5.09705e-04 & 1.01  \\
\cline{3-10}
 & & $\|u-u_h\|_{L^\infty}$  & 7.48743e-04 & 3.67246e-04 & 1.03 & 1.81917e-04 & 1.01 & 9.05192e-05 & 1.01  \\
\hline
\multirow{2}{*}{(\ref{ch_FPDGFull1st+})}  & \multirow{2}{*}{$64^2$} & $\|u-u_h\|_{L^2}$ & 4.21016e-03 & 2.06606e-03 & 1.03 & 1.02364e-03 & 1.01 & 5.09522e-04 & 1.01  \\
\cline{3-10}
& & $\|u-u_h\|_{L^\infty}$  & 7.48925e-04 & 3.67359e-04 & 1.03 & 1.82003e-04 & 1.01 & 9.05934e-05 & 1.01  \\
\hline
\hline
\multirow{2}{*}{(\ref{FPDGFull+})}  & \multirow{2}{*}{$64^2$} & $\|u-u_h\|_{L^2}$ & 1.32995e-03 & 3.18993e-04 & 2.06 & 7.69932e-05 & 2.05 & 1.88763e-05 & 2.03  \\
\cline{3-10}
& & $\|u-u_h\|_{L^\infty}$  & 2.24427e-04 & 5.35331e-05 & 2.07 & 1.27796e-05 & 2.07 & 3.12208e-06 & 2.03  \\
\hline
\multirow{2}{*}{(\ref{FPDGFull+})}  & \multirow{2}{*}{$128^2$} & $\|u-u_h\|_{L^2}$ & 1.34130e-03 & 3.18230e-04 & 2.08 & 7.69137e-05 & 2.05 & 1.88416e-05 & 2.03  \\
\cline{3-10}
& & $\|u-u_h\|_{L^\infty}$  & 2.27726e-04 & 5.31968e-05 & 2.10 & 1.27386e-05 & 2.06 & 3.09279e-06 & 2.04  \\
 \hline
\end{tabular}\label{timeacc}
\end{table}
\end{example}

\begin{example} Following \cite{WKG06}, we consider the Cahn-Hilliard equation (\ref{CH}) with constant mobility $M(u)=1$, the logarithmic Flory-Huggins potential
$$
F(u) = 600\left(u\ln u + (1-u) \ln (1-u)\right) + 1800u(1-u),
$$
and the parameters $\epsilon=1$ and $B=10^2$. The equation is subject to the initial condition
$$
u_0(x,y)=
\left \{
\begin{array}{rl}
0.71, \quad & (x,y) \in \Omega_1, \\
0.69, \quad & (x,y) \in \Omega_2, \\
\end{array}
\right.
$$
where the square domain
$$
\Omega = [-0.5, 0.5]\times [-0.5, 0.5], \quad \Omega_1 = [-0.2, 0.2]\times [-0.2, 0.2], \quad \Omega_2 = \Omega \backslash \Omega_1.
$$
The boundary conditions are taken as Neumann BCs, (ii) in (\ref{BC}).

\noindent\textbf{Test case 1.}
We first solve this problem by the first order fully discrete IEQ-DG scheme (\ref{ch_FPDGFull1st+}) based on $P^1$ and $P^2$ polynomials with time step $\Delta t=10^{-7}$ and meshes $40\times40$ and $80\times80$, respectively.
The contours at $T=8 \times 10^{-5}$ are shown in Figure \ref{circlecontour}, and the corresponding energy and mass evolutions are shown in Figure \ref{circleengmass}. From Figure \ref{circlecontour}, we find that the solution structure is well resolved even on coarser mesh and lower order $P^1$ polynomials, and the scheme (\ref{ch_FPDGFull1st+}) using $P^2$ polynomials gives a better resolution than that using $P^1$ polynomials on coarser meshes $40\times40$, but there is no noticeable difference with solution on refined meshes $80\times80$ or higher order polynomial $P^2$ as shown in Figure \ref{circlecontour}(b)-(d). The pattern structure is well consistent with that obtained in \cite{WKG06}. Figure \ref{circleengmass}(a) shows that the numerical solution of the scheme (\ref{ch_FPDGFull1st+}) satisfies the energy dissipation law, Figure \ref{circleengmass}(b) and \ref{circleengmass}(c) show that the numerical solution conserves the total mass $\int_\Omega u_h^n dx=0.6932$ under an appropriate tolerance.

\begin{figure}
\centering
\subfigure[]{\includegraphics[width=0.49\textwidth]{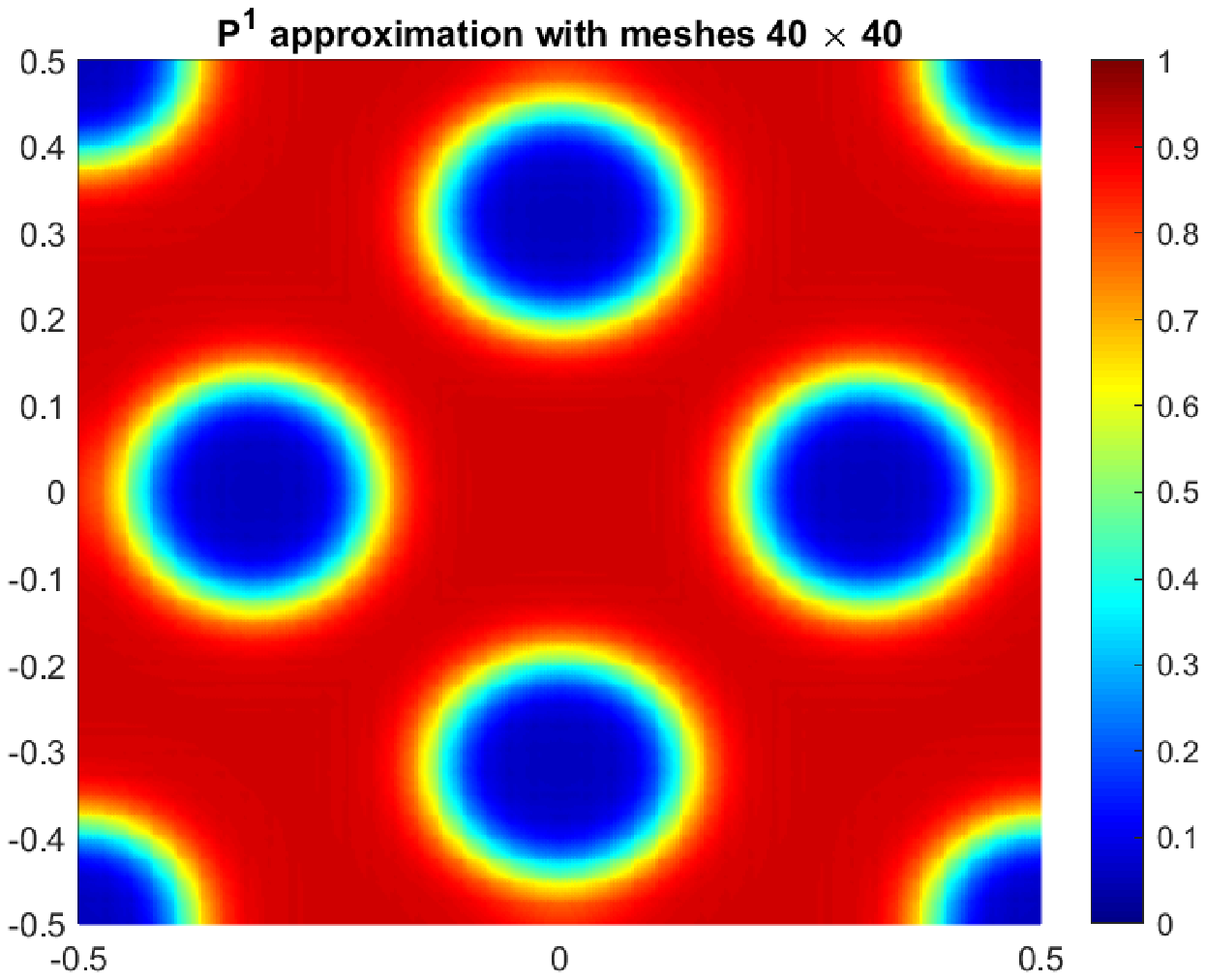}}
\subfigure[]{\includegraphics[width=0.49\textwidth]{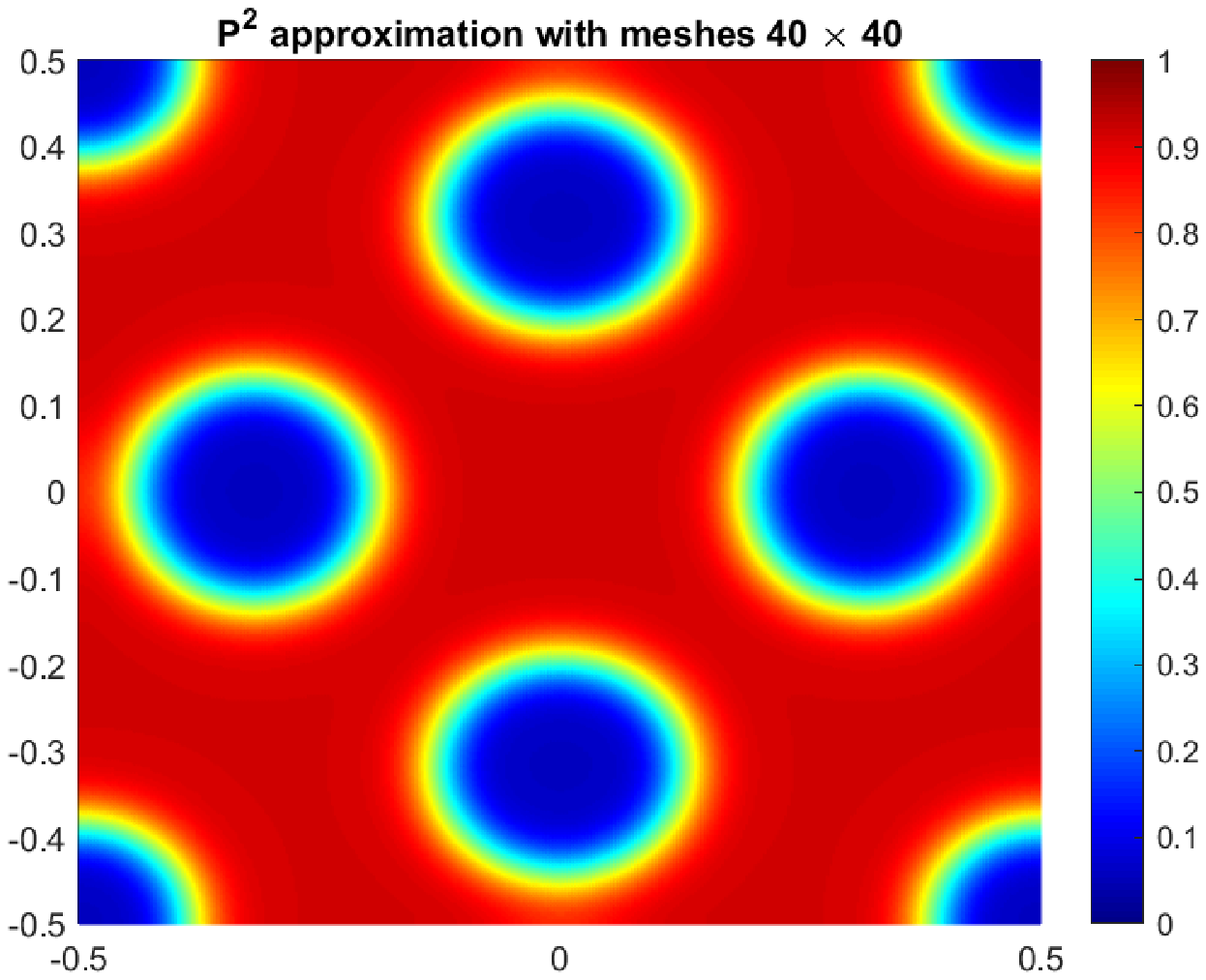}}
\subfigure[]{\includegraphics[width=0.49\textwidth]{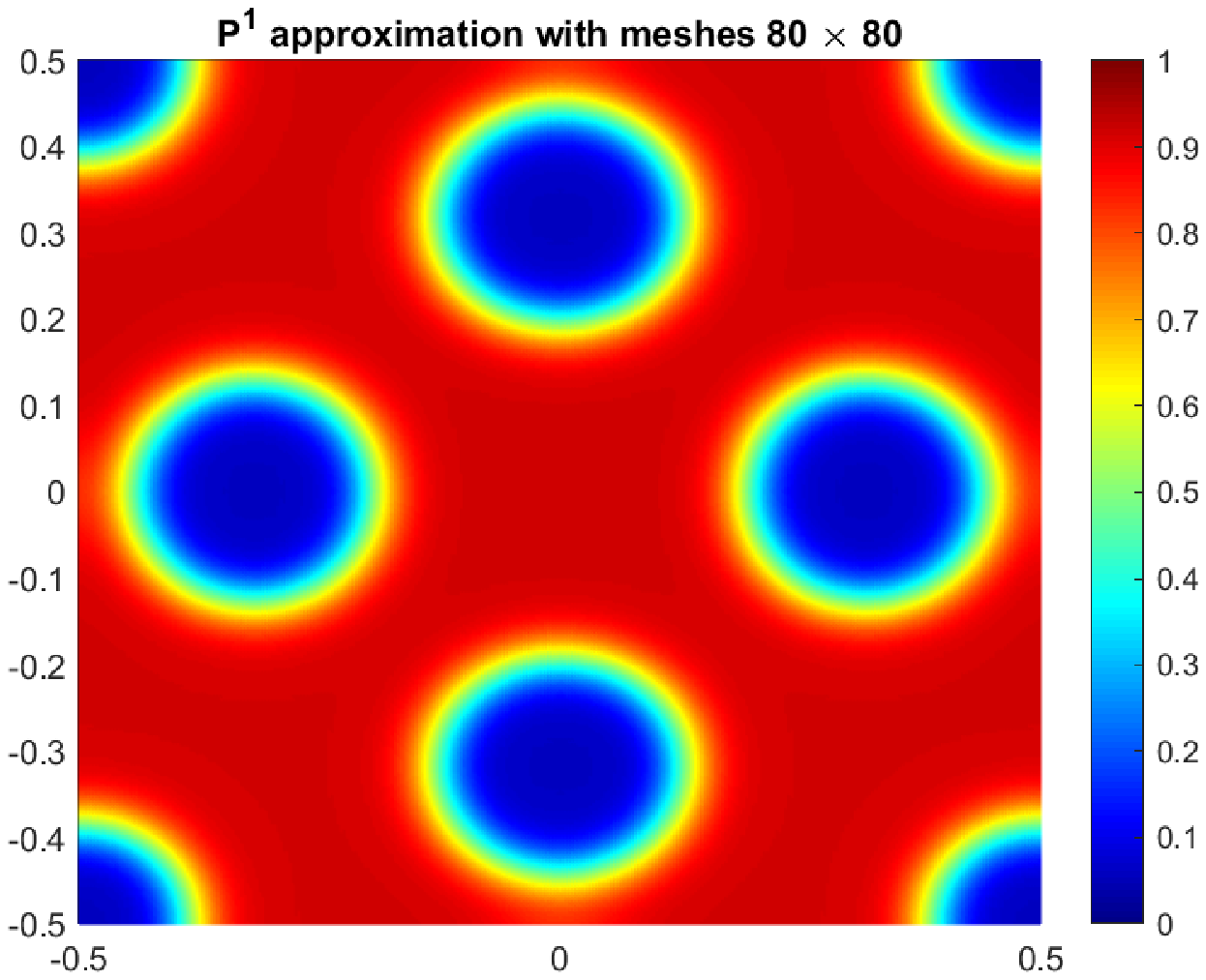}}
\subfigure[]{\includegraphics[width=0.49\textwidth]{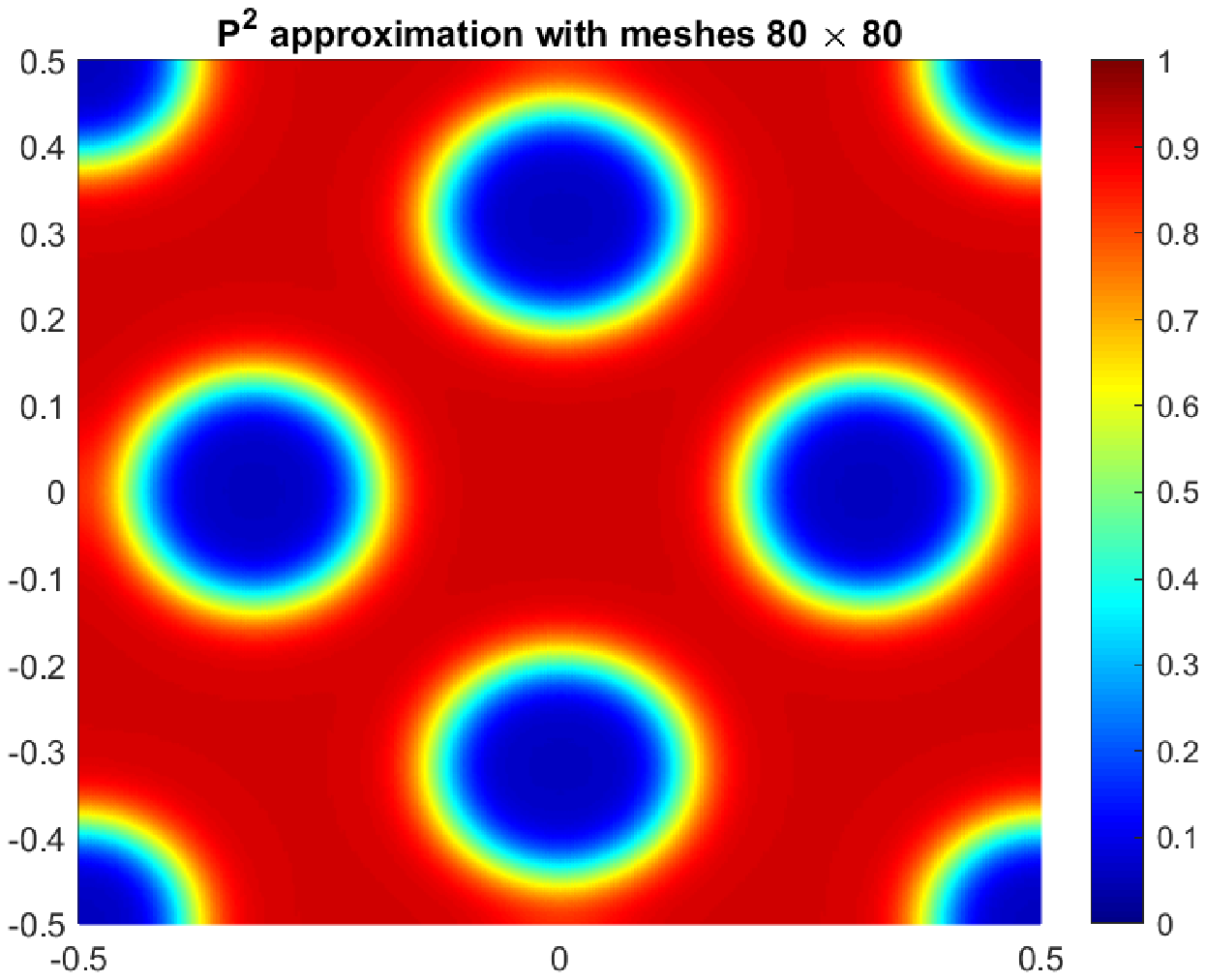}}
\caption{The contours of numerical solution for the scheme (\ref{ch_FPDGFull1st+}).}\label{circlecontour}
\end{figure}

\begin{figure}
\centering
\subfigure[]{\includegraphics[width=0.325\textwidth]{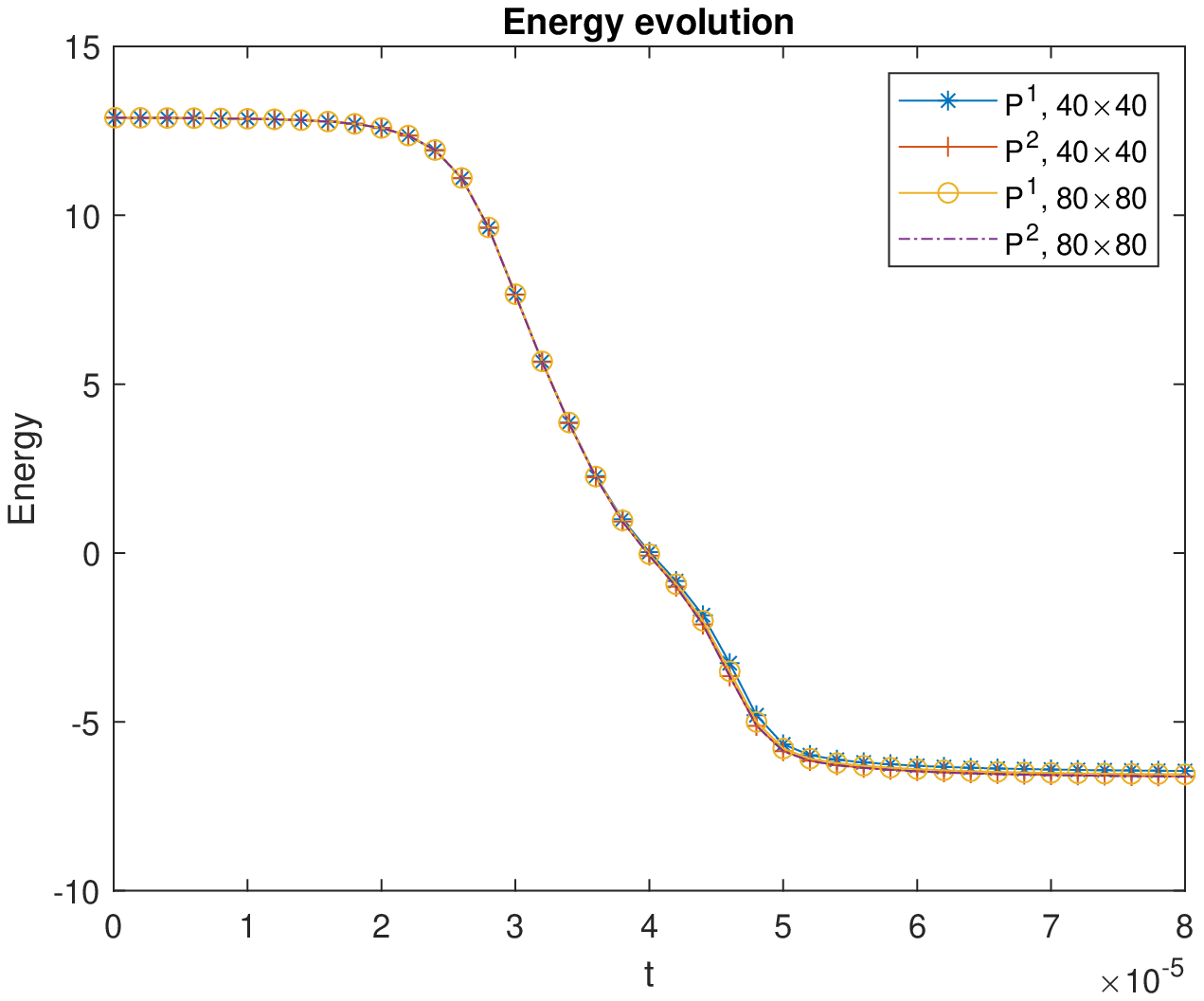}}
\subfigure[]{\includegraphics[width=0.325\textwidth]{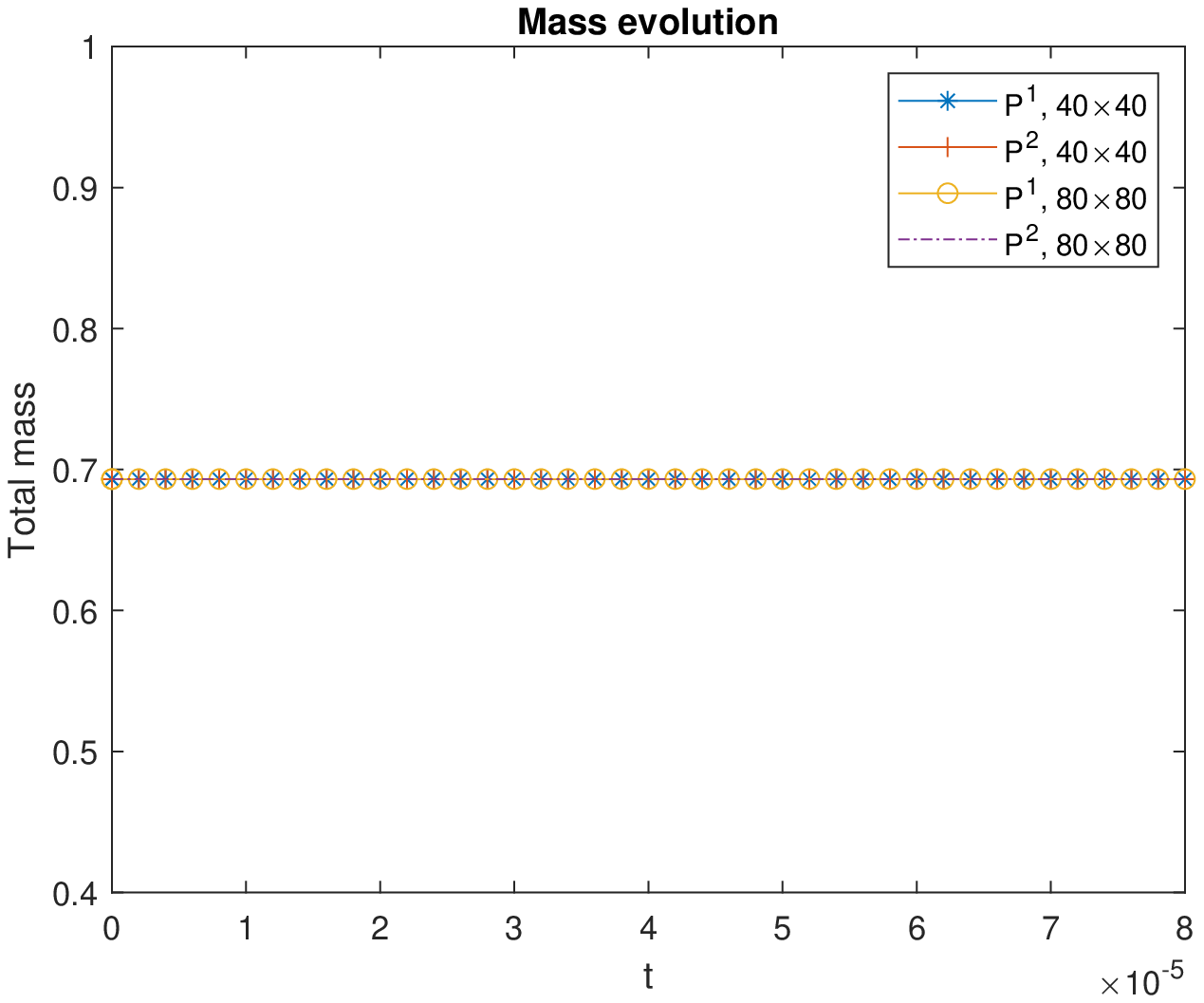}}
\subfigure[]{\includegraphics[width=0.325\textwidth]{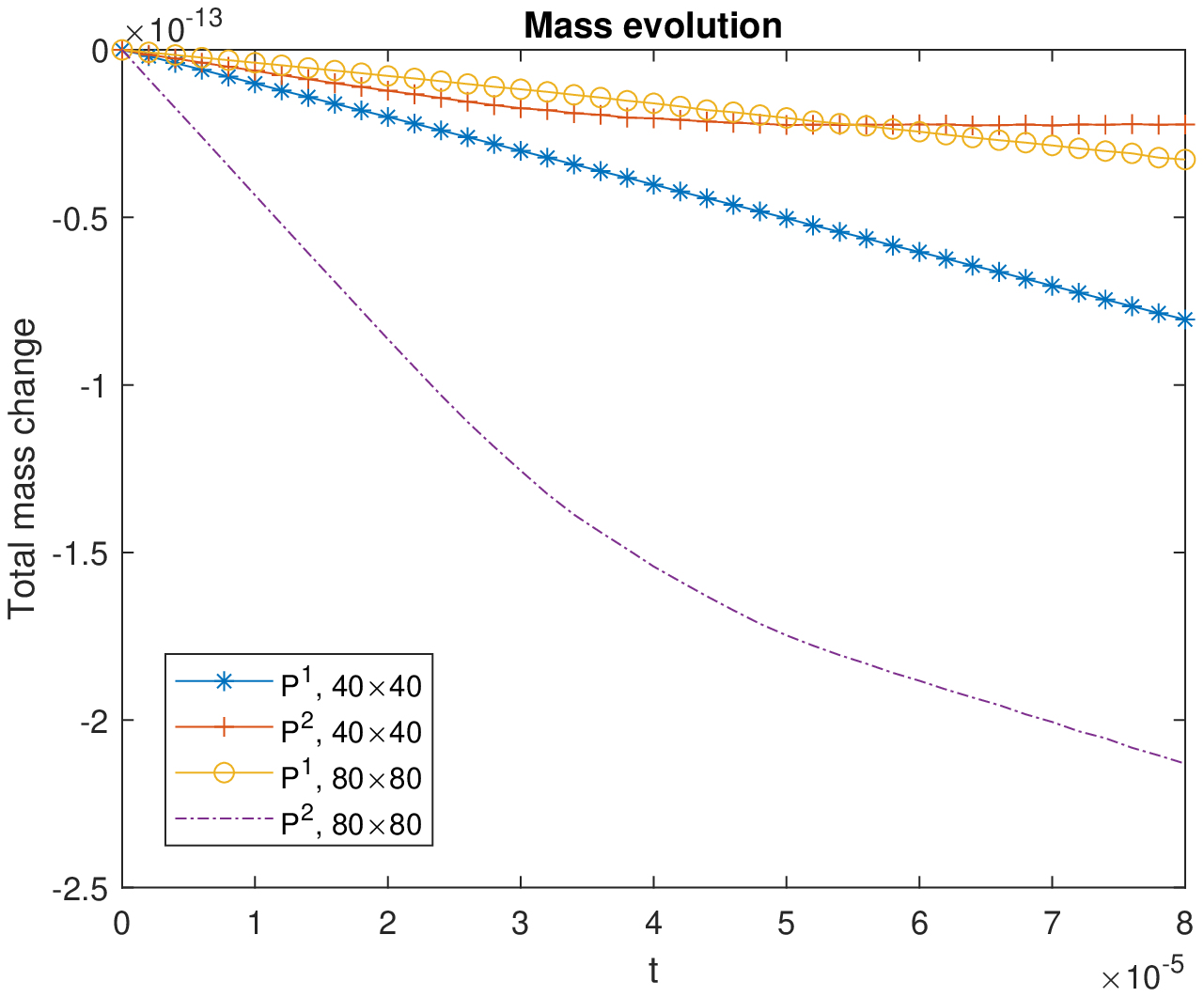}}
\caption{The energy and mass evolution of numerical solution for the scheme (\ref{ch_FPDGFull1st+}). (a) The energy and mass evolution; (b) the total mass evolution ; (c) the evolution of the total mass difference $\int_{\Omega}u_h^n-u_h^0dx$.}\label{circleengmass}
\end{figure}

\noindent\textbf{Test case 2.}
We solve this problem again by the second order fully discrete IEQ-DG scheme (\ref{FPDGFull+}) based on $P^1$ and $P^2$ polynomials. In Figure \ref{circlecontour2}, we show the contours at $T=8 \times 10^{-5}$ obtained based on $P^1$ polynomials with mesh $40\times40$ and time steps  $\Delta t=10^{-7}, 8\times 10^{-8}, 5\times 10^{-8}, 2\times 10^{-8}$, respectively. From Figure \ref{circlecontour2}, we find the pattern structure is comparable to that in Figure \ref{circlecontour}(b)-(d) even with time step $\Delta t=10^{-7}$ and lower order $P^1$ polynomials.


Figure \ref{circleengmass2} shows that the numerical solution of the scheme (\ref{FPDGFull+}) satisfies the energy dissipation law (\ref{engdis}), but we do find that the modified energy (\ref{menergy}) better approximate the original energy with a smaller time step $\Delta t$, a smaller mesh size $h$ or polynomials of a higher degree. Figure \ref{circleengmass2+} implies the numerical solutions with different time steps $\Delta t$ conserve the total mass $\int_\Omega u dx=0.6932$ under an appropriate tolerance.

\begin{figure}
\centering
\subfigure[]{\includegraphics[width=0.49\textwidth]{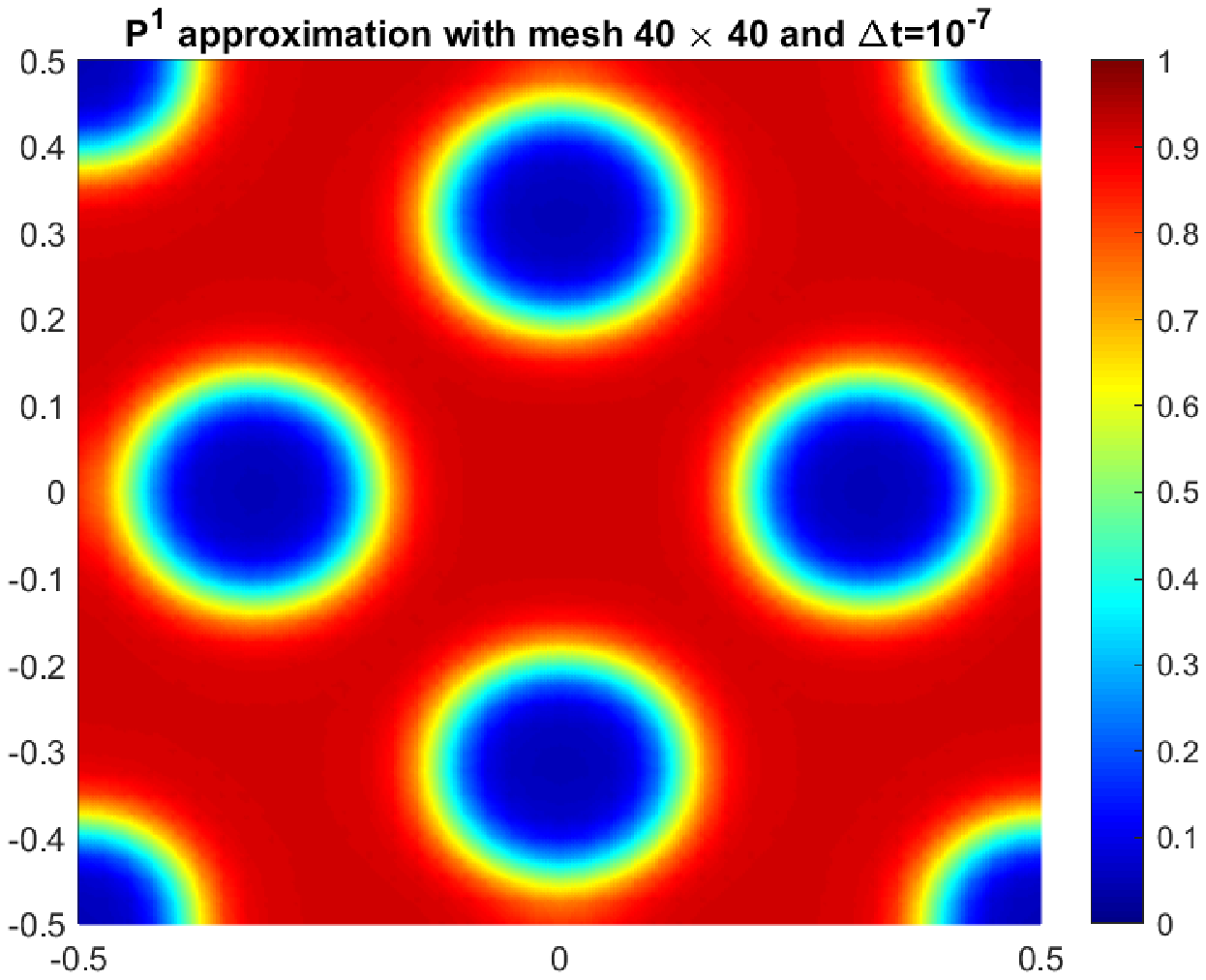}}
\subfigure[]{\includegraphics[width=0.49\textwidth]{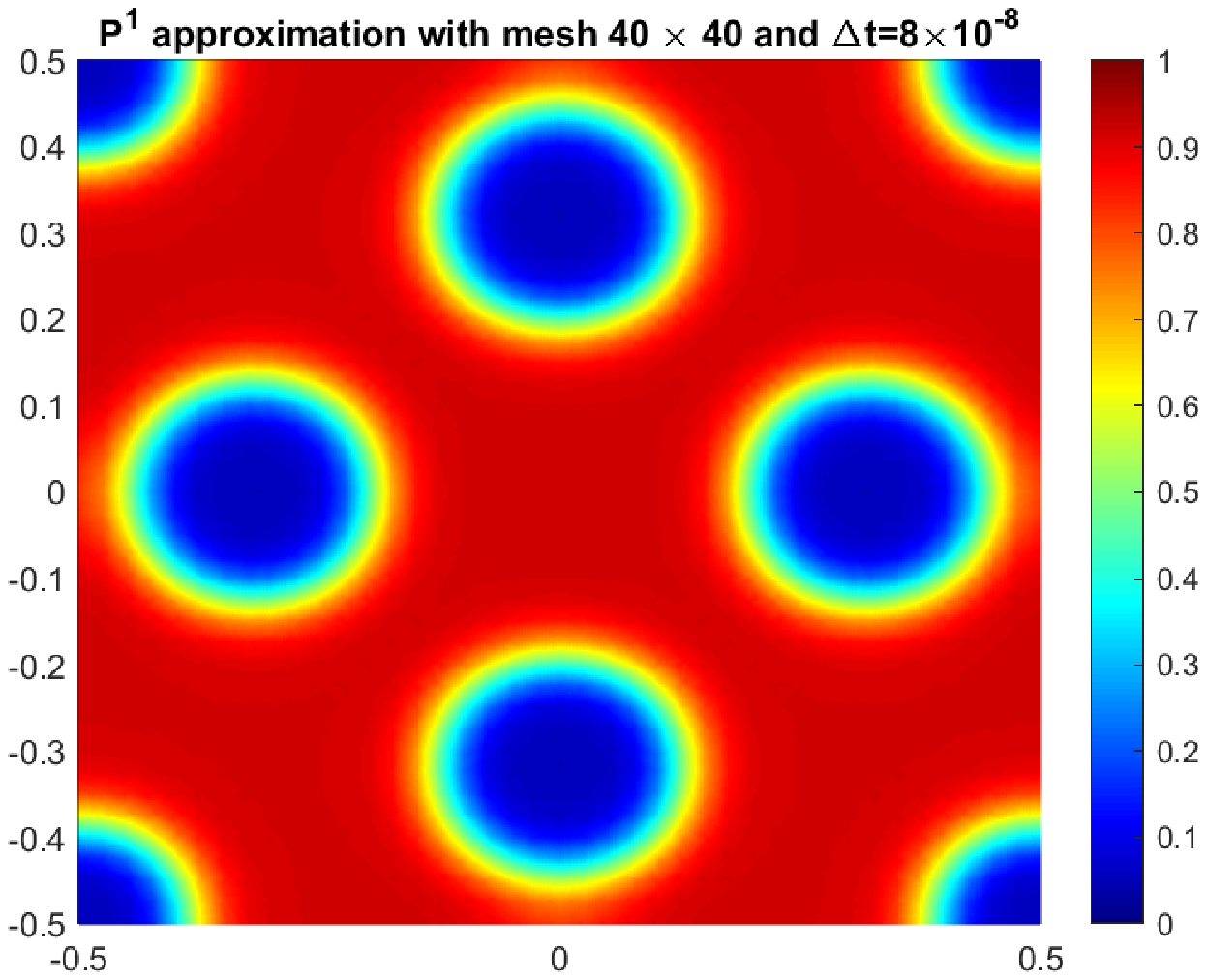}}
\subfigure[]{\includegraphics[width=0.49\textwidth]{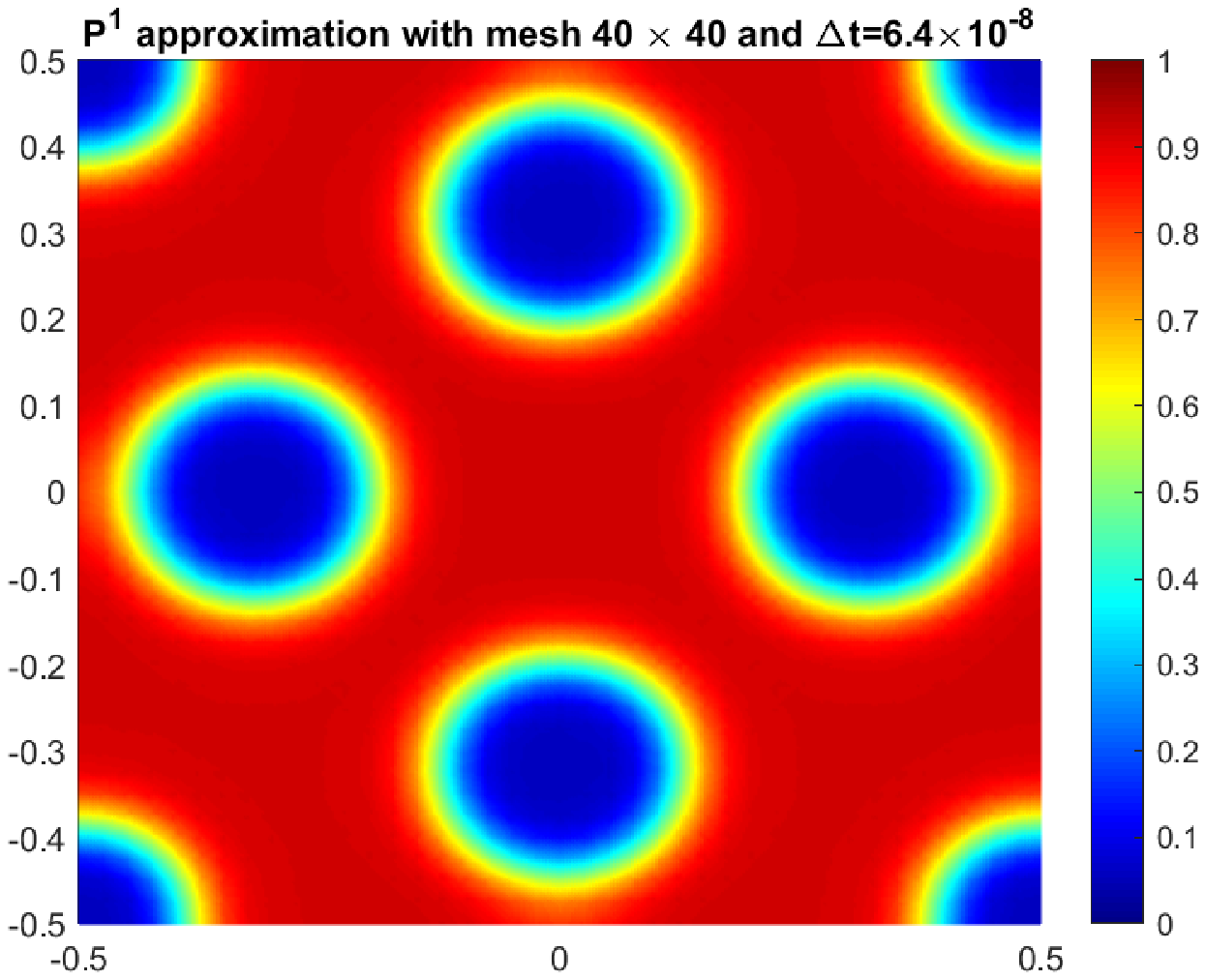}}
\subfigure[]{\includegraphics[width=0.49\textwidth]{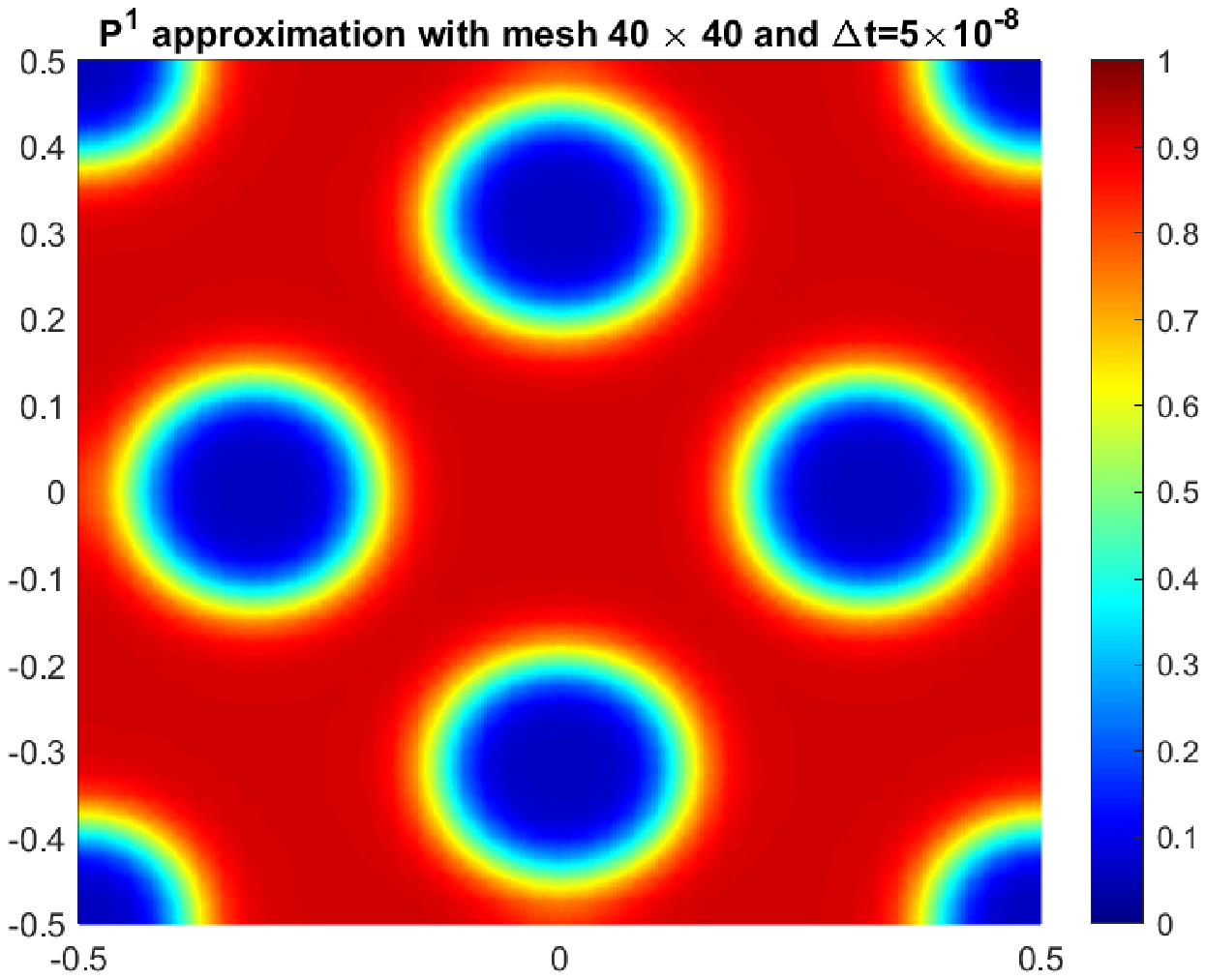}}
\caption{The contours of numerical solution for the scheme (\ref{FPDGFull+}).}\label{circlecontour2}
\end{figure}

\begin{figure}
\centering
\subfigure[]{\includegraphics[width=0.325\textwidth]{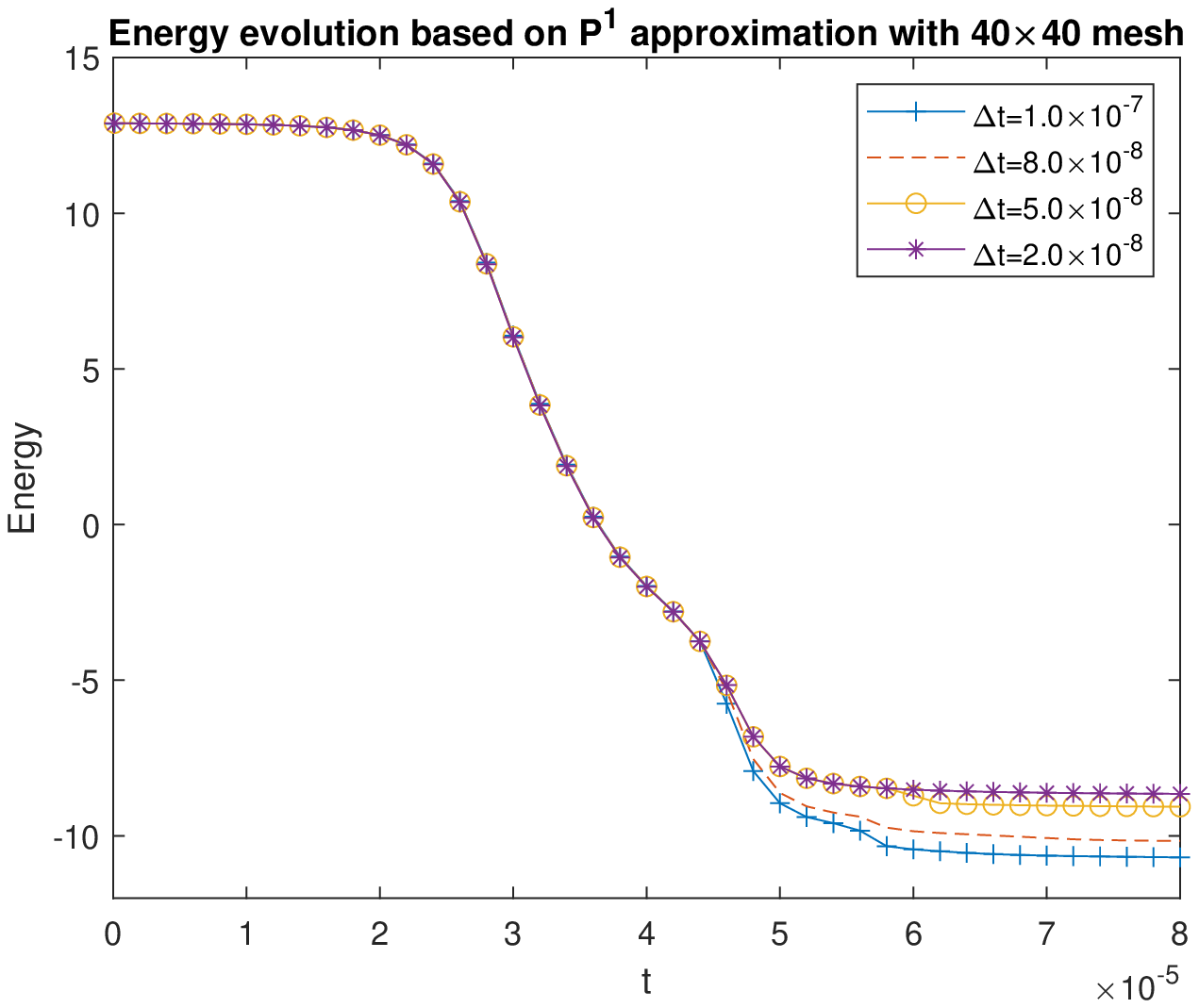}}
\subfigure[]{\includegraphics[width=0.325\textwidth]{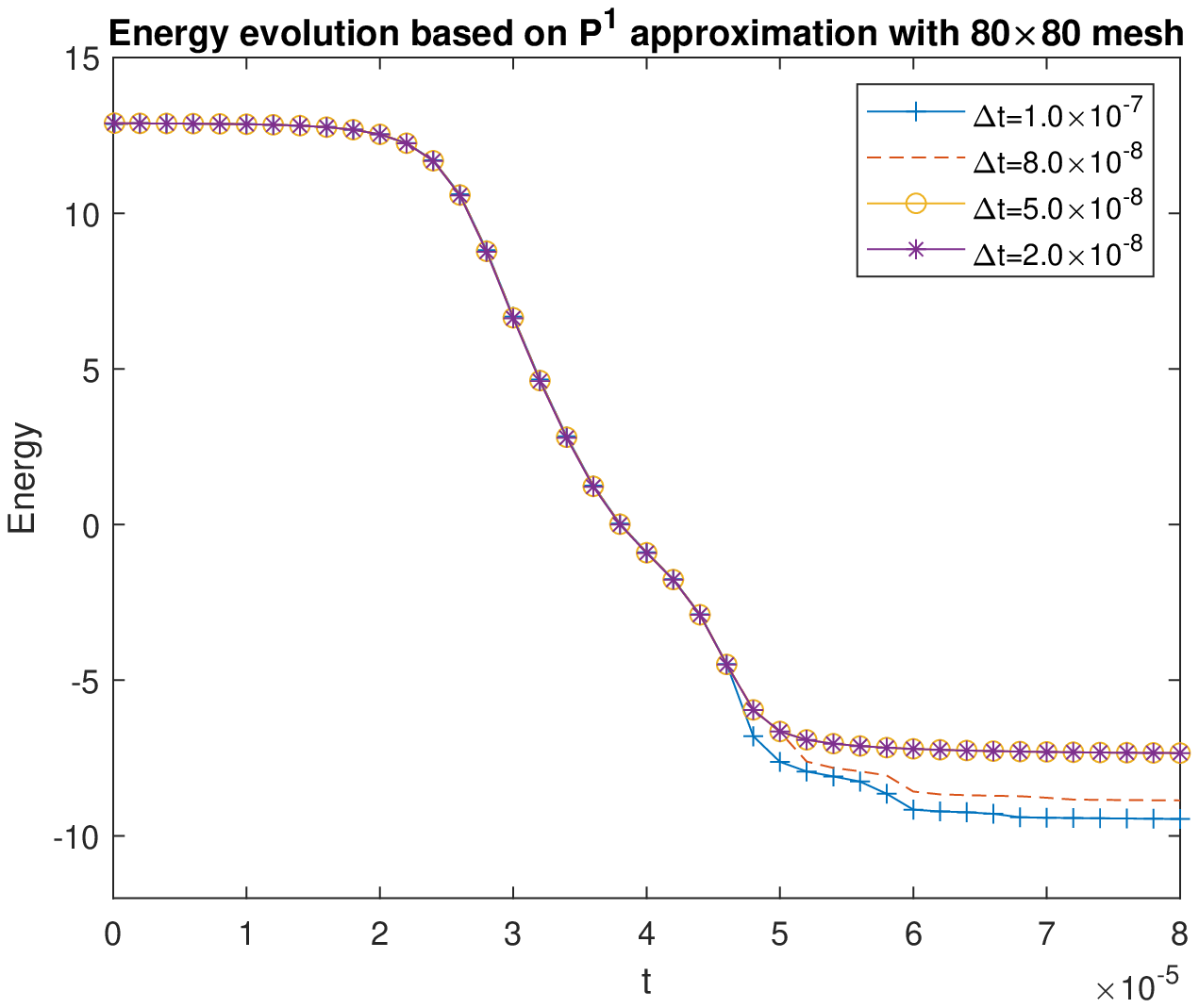}}
\subfigure[]{\includegraphics[width=0.325\textwidth]{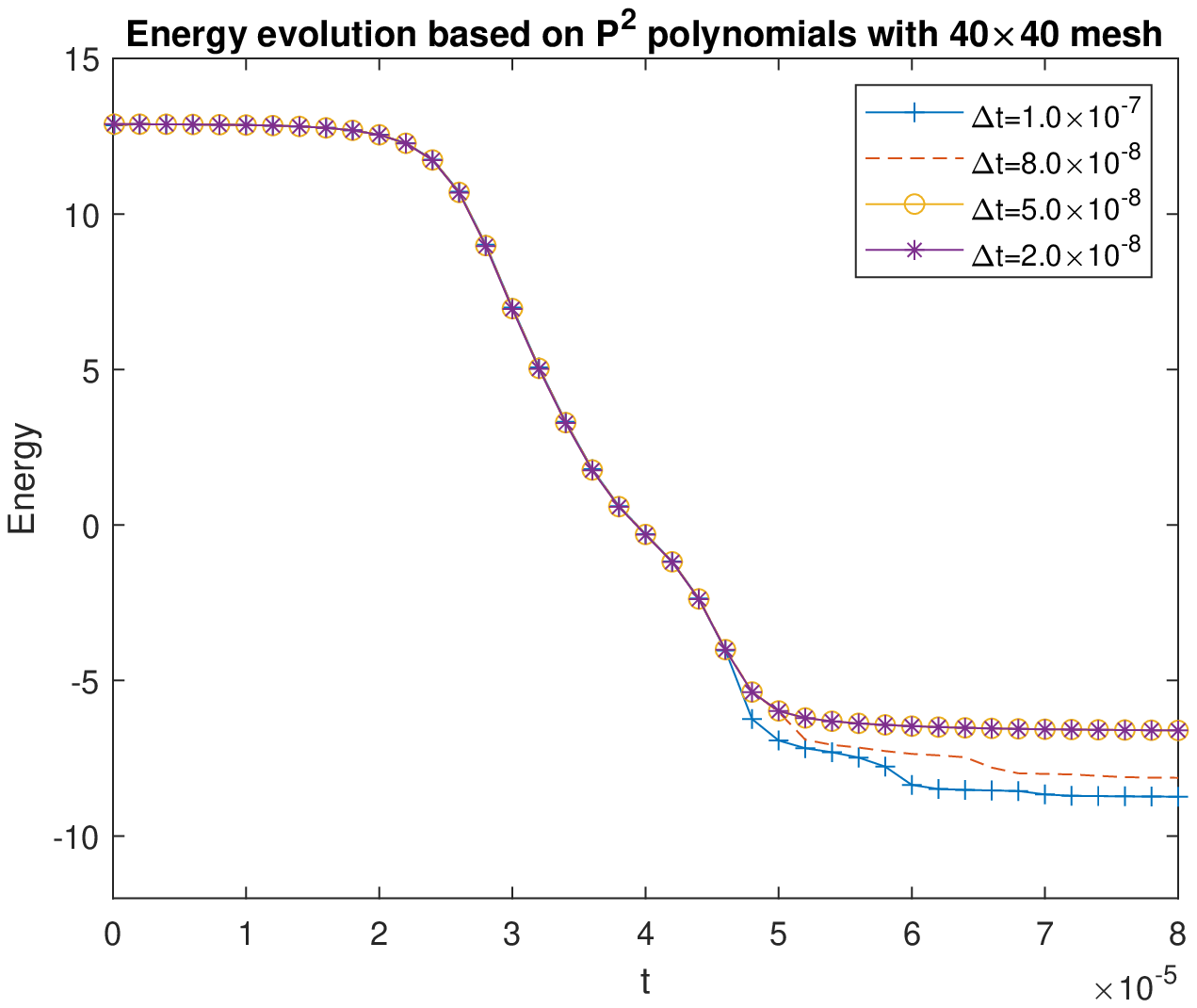}}
\caption{The energy evolution of numerical solution for the scheme (\ref{FPDGFull+}).}\label{circleengmass2}
\end{figure}

\begin{figure}
\centering
\subfigure[]{\includegraphics[width=0.325\textwidth]{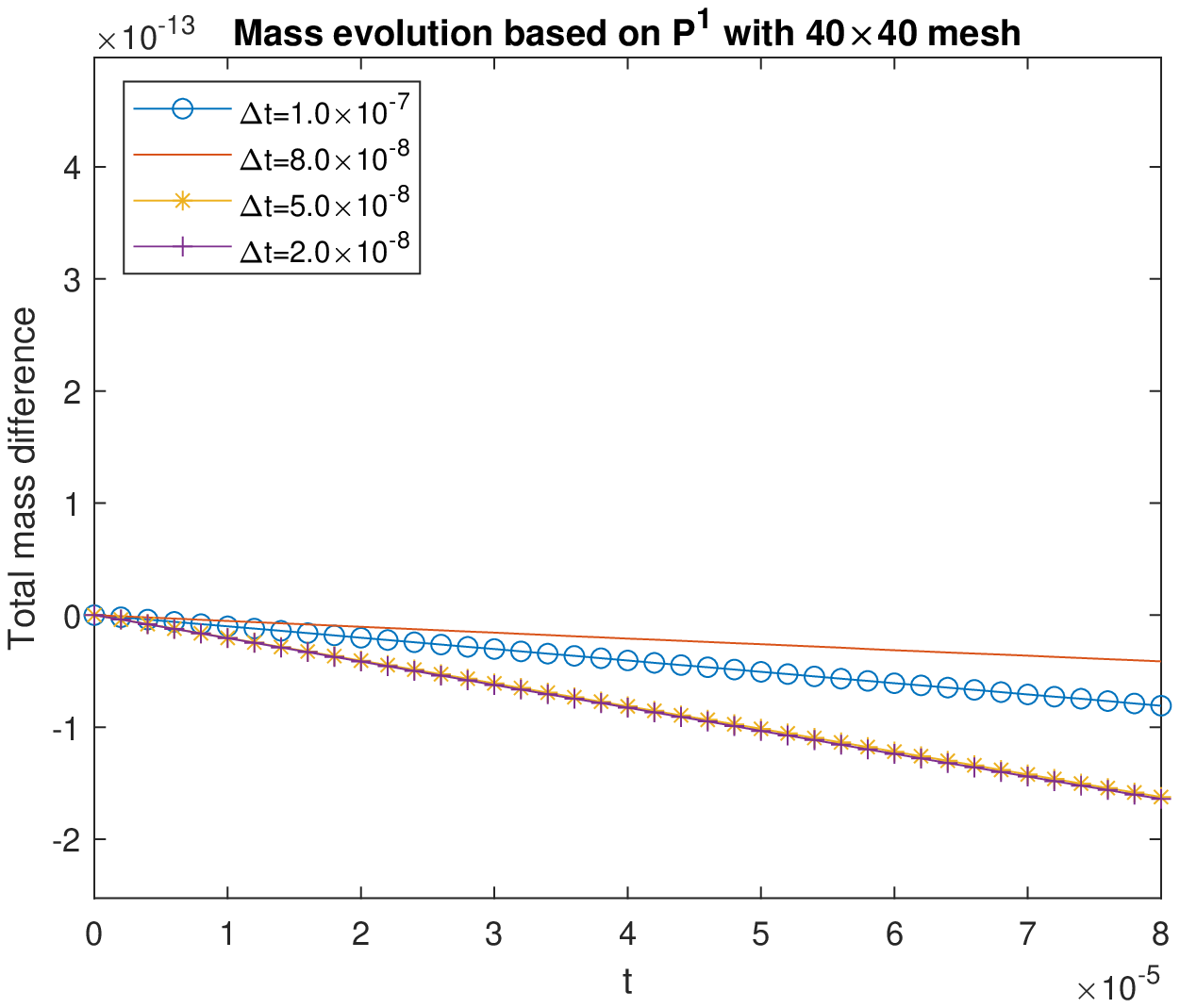}}
\subfigure[]{\includegraphics[width=0.325\textwidth]{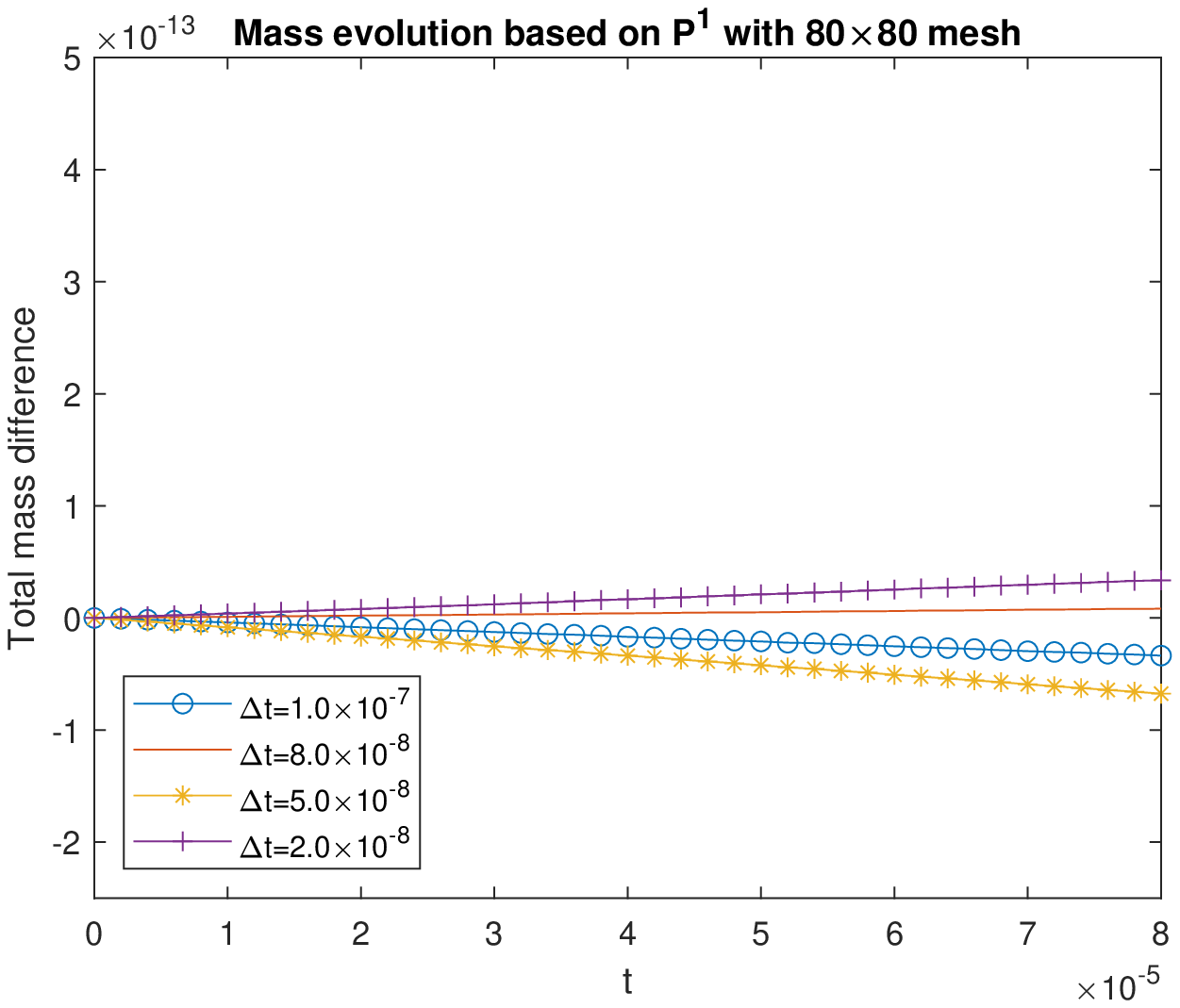}}
\subfigure[]{\includegraphics[width=0.325\textwidth]{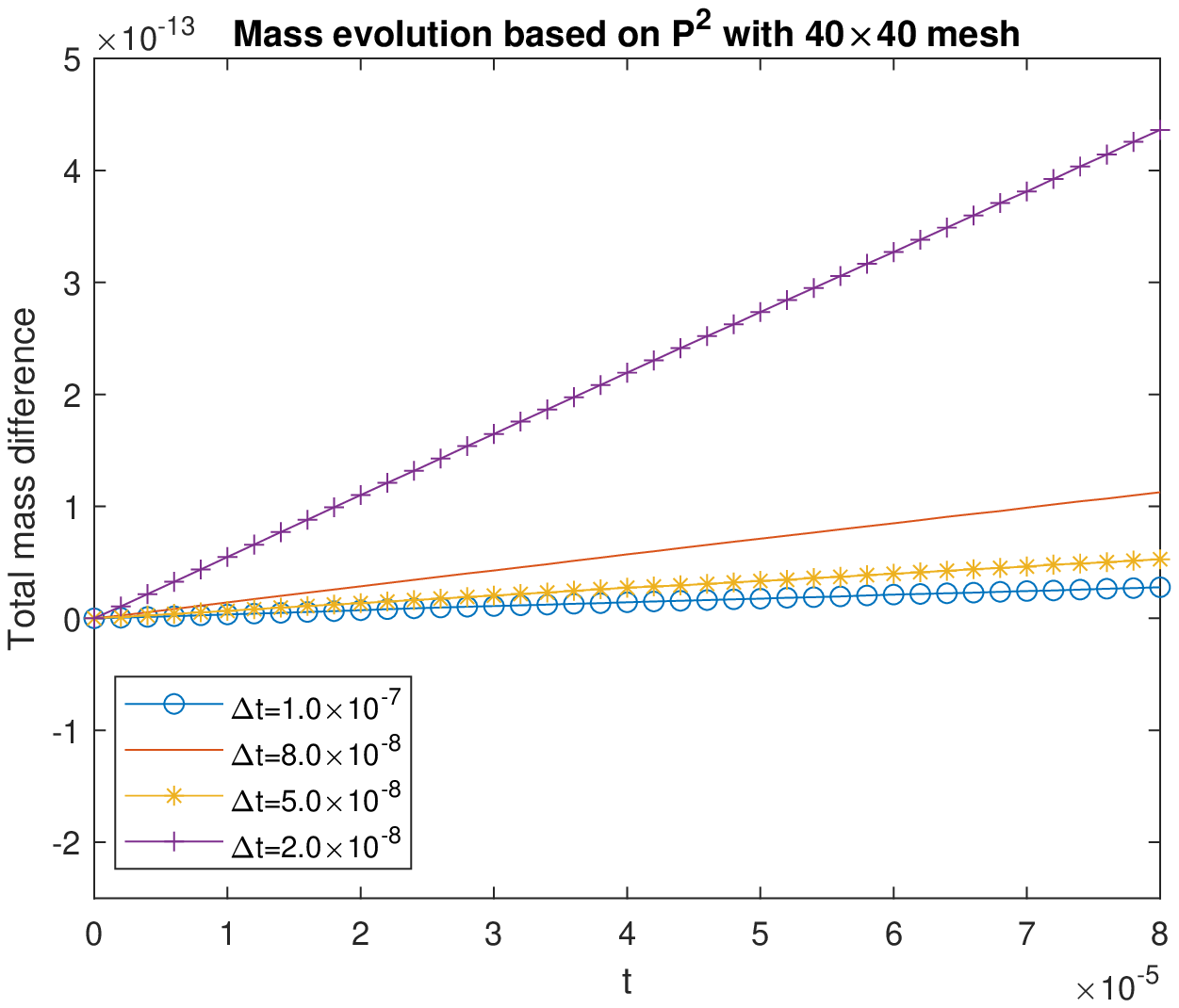}}
\caption{The total mass difference $\int_{\Omega}u_h^n-u_h^0dx$ evolution for scheme (\ref{FPDGFull+}).}\label{circleengmass2+}
\end{figure}

\end{example}

\begin{example} Following \cite{WKG06}, we further consider the Cahn-Hilliard equation (\ref{CH}) with degenerate mobility $M(u)=u(1-u)$, the logarithmic Flory-Huggins potential
$$
F(u) = 3000\left(u\ln u + (1-u) \ln (1-u)\right) + 9000u(1-u),
$$
and the parameters $\epsilon=1$ and $B=10^3$. The initial condition is
$$
u_0(x,y) = 0.63 + 0.05 \text{rand}(x,y),
$$
where $\text{rand}(x,y)$ is the random perturbation function in $[-1,1]$ and has zero mean. For the boundary conditions, we take Neumann BCs (ii) in (\ref{BC}).

We solve this problem by the scheme (\ref{FPDGFull+}) based on $P^2$ polynomials with meshes $64\times64$ and time step
$\Delta t=10^{-8}$. The evolution of the concentration field is shown in Figure \ref{Npat}. The corresponding energy and mass evolutions are shown in Figure \ref{Nmasseng}. Figure \ref{Npat} clearly shows the
two phases of the concentration evolution. The first phase is governed by spinodal decomposition and phase separation, which is roughly corresponding to the first three figures of Figure \ref{Npat},
this period is basically terminated as soon as the local concentration is driven to either value of the two binodal points.
The second phase is governed by grain coarsening, approximately from $t=8\times 10^{-6}$ onwards the generated patterns cluster and grains tend to coarsen, which is a very slow process. Figure \ref{Npat} shows statistically similar patterns in the numerical solution as those in \cite{WKG06}. Figure \ref{Nmasseng} further confirms the numerical solution of the scheme (\ref{FPDGFull+}) satisfies the energy dissipation law and conserves the total mass $\int_\Omega u dx=0.63$.
\begin{figure}
\centering
\subfigure{\includegraphics[width=0.325\textwidth]{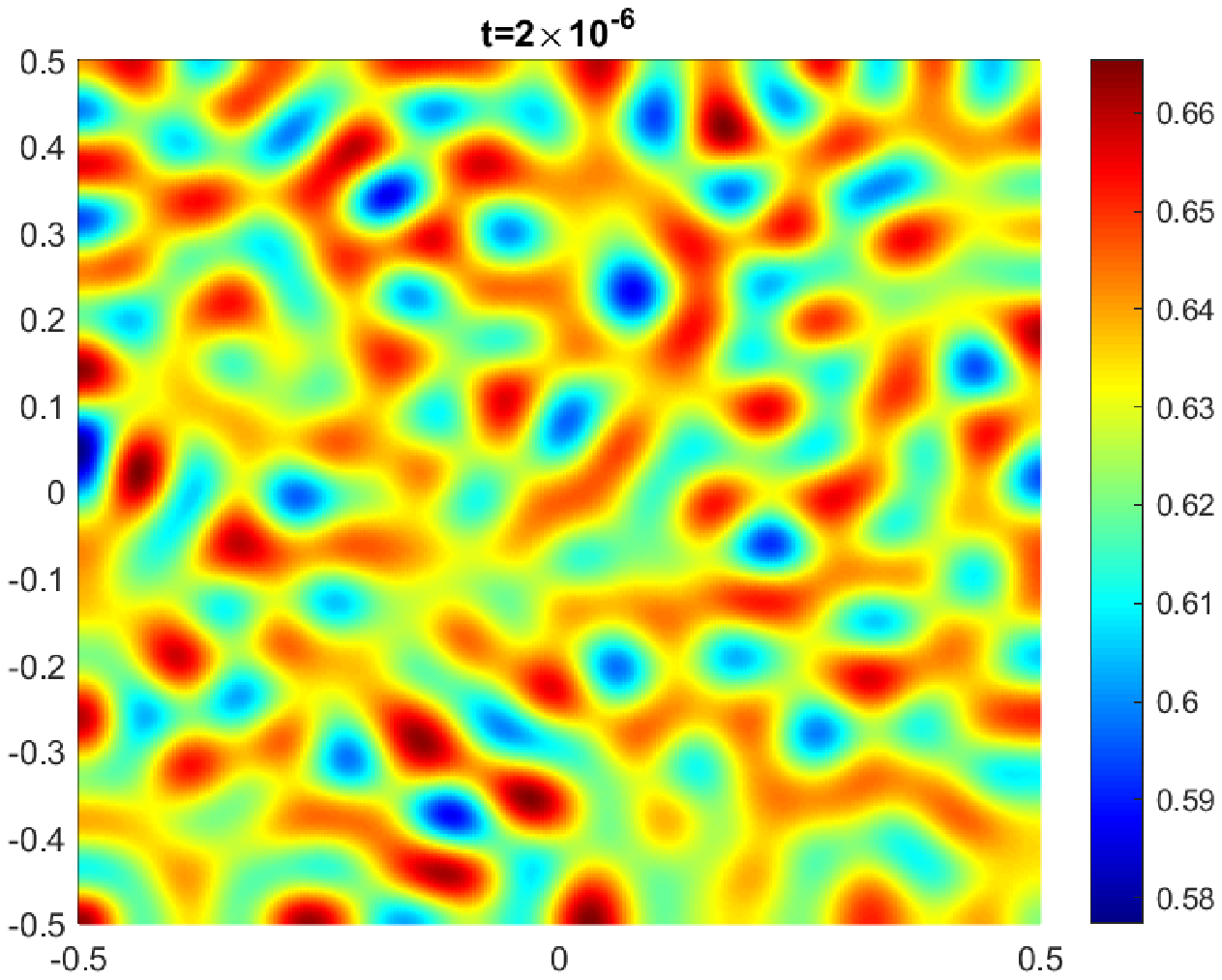}}
\subfigure{\includegraphics[width=0.325\textwidth]{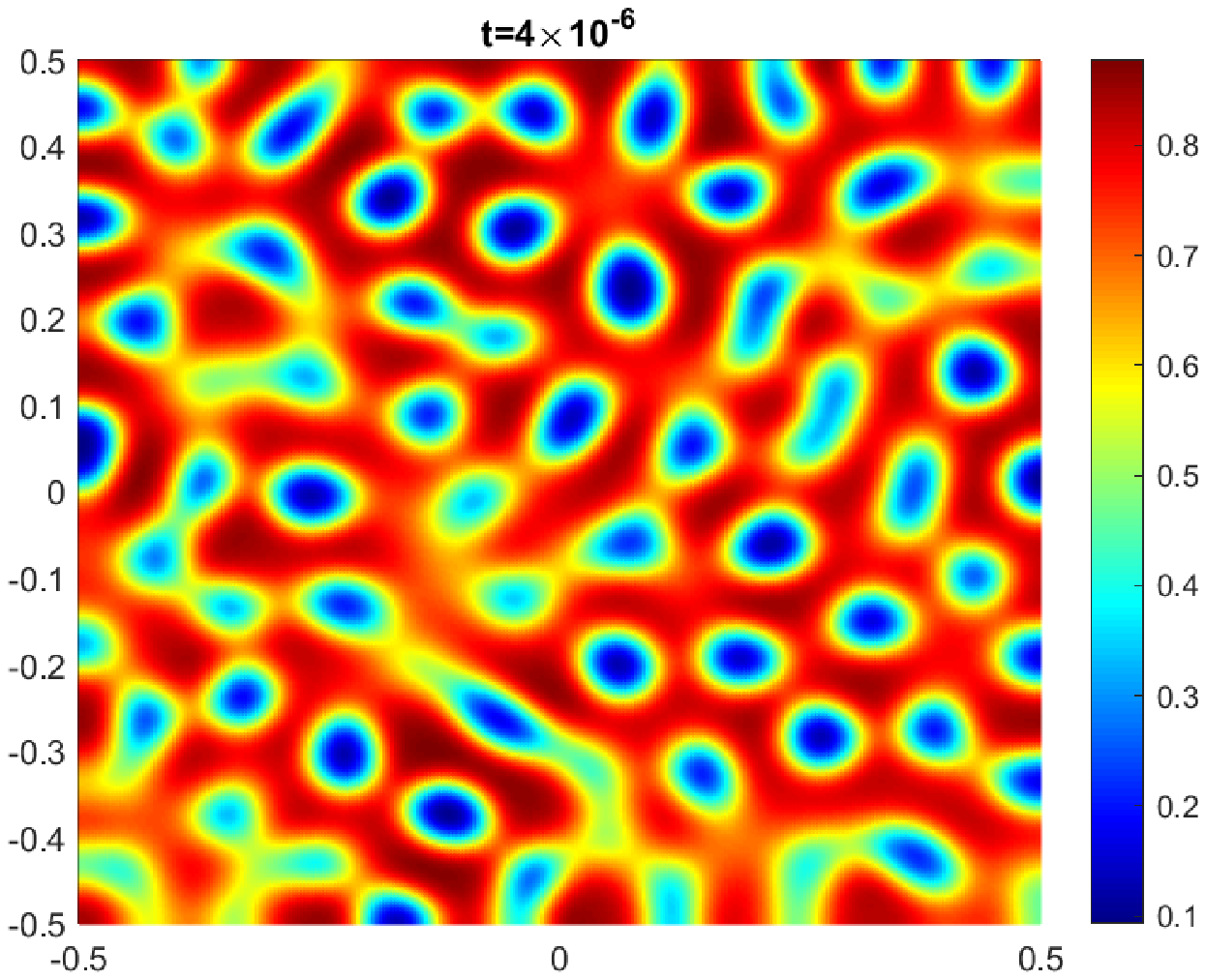}}
\subfigure{\includegraphics[width=0.325\textwidth]{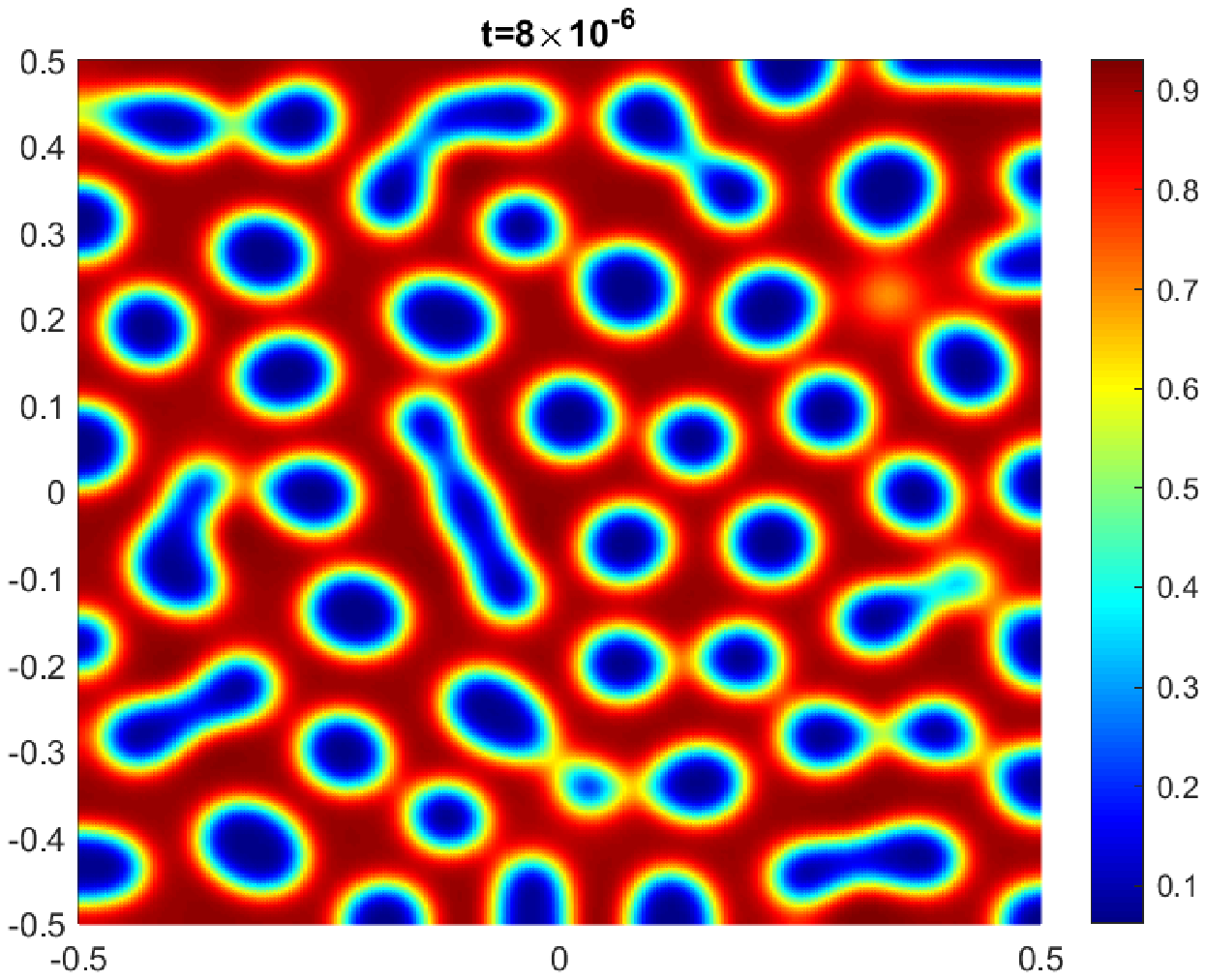}}
\subfigure{\includegraphics[width=0.325\textwidth]{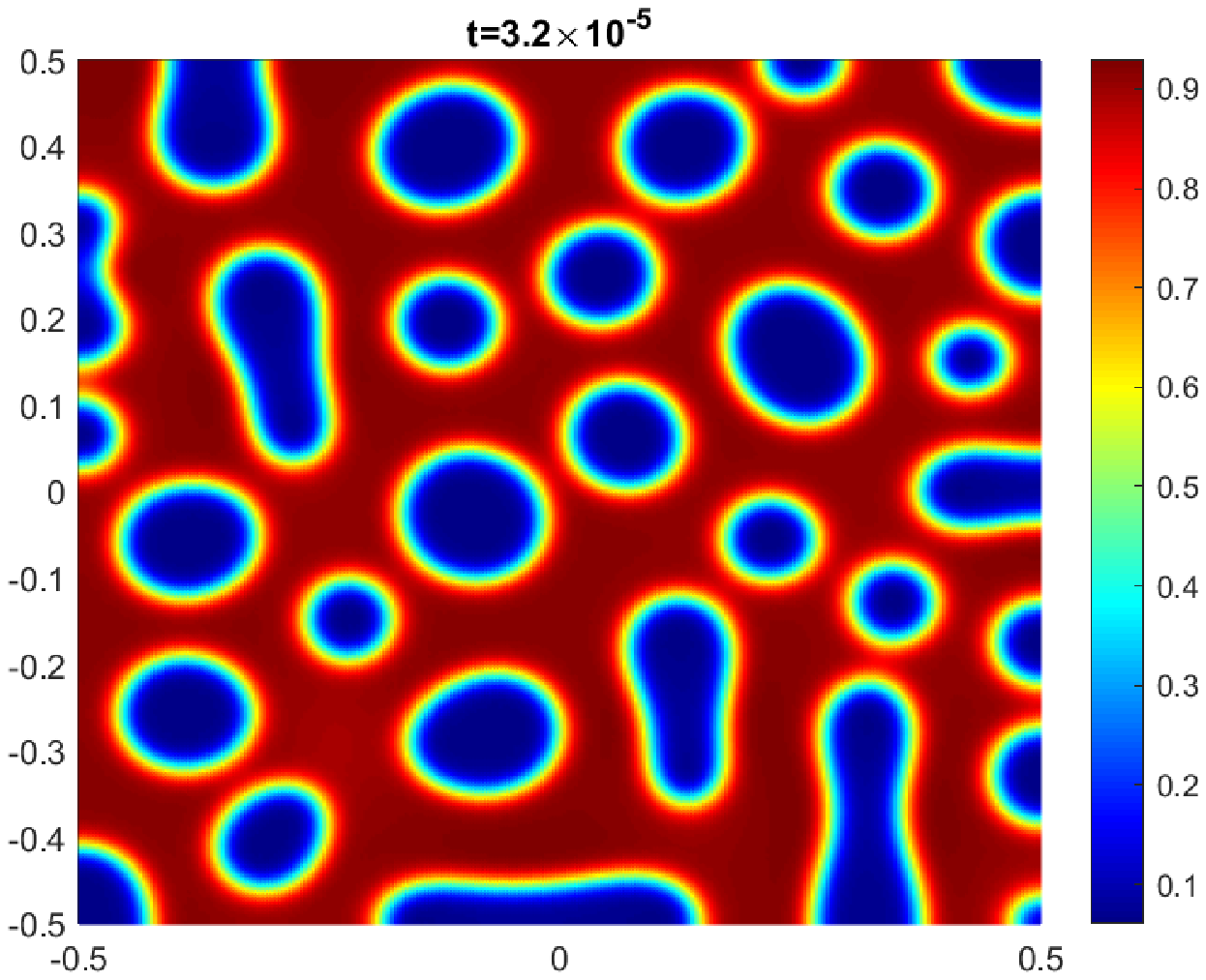}}
\subfigure{\includegraphics[width=0.325\textwidth]{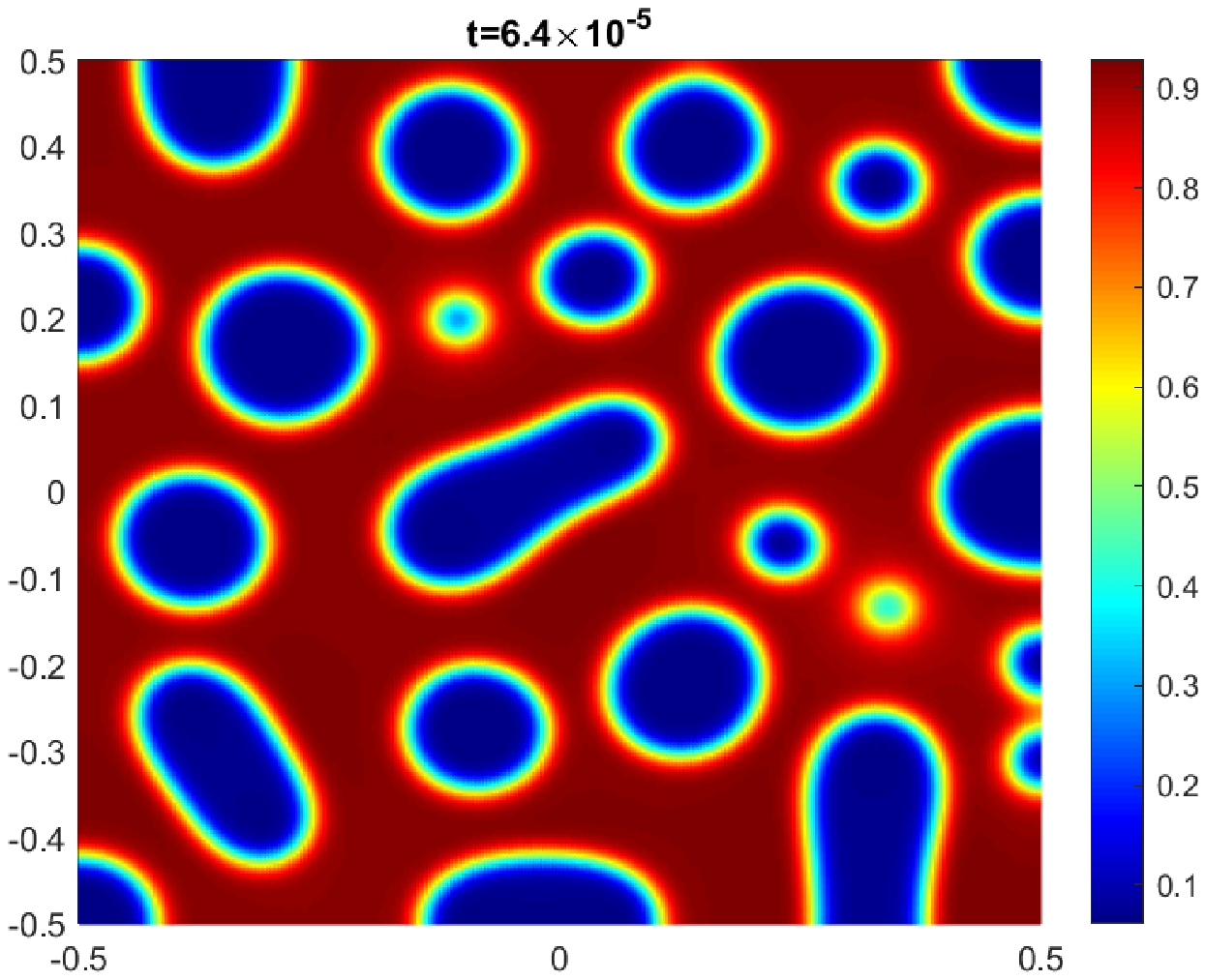}}
\subfigure{\includegraphics[width=0.325\textwidth]{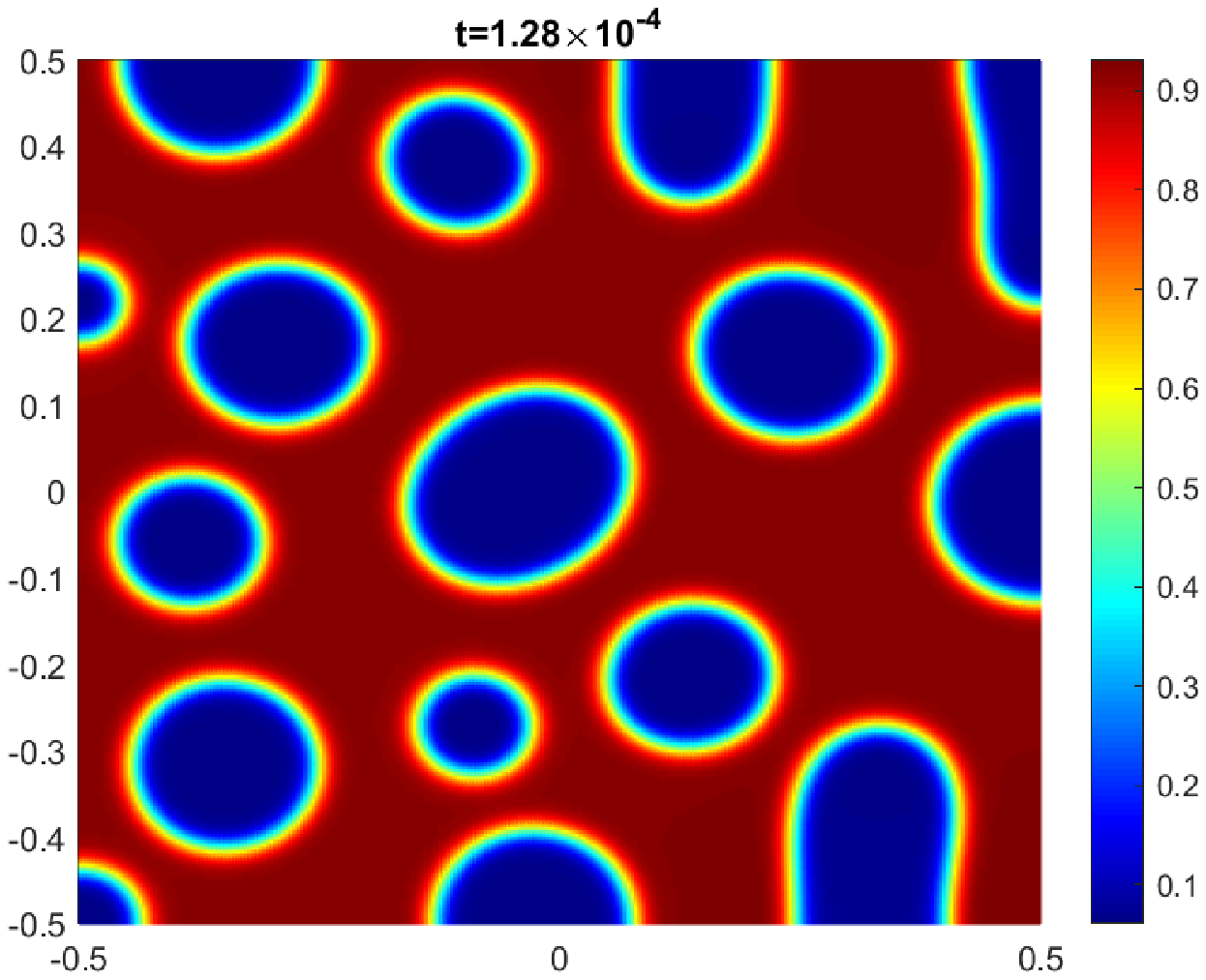}}
\subfigure{\includegraphics[width=0.325\textwidth]{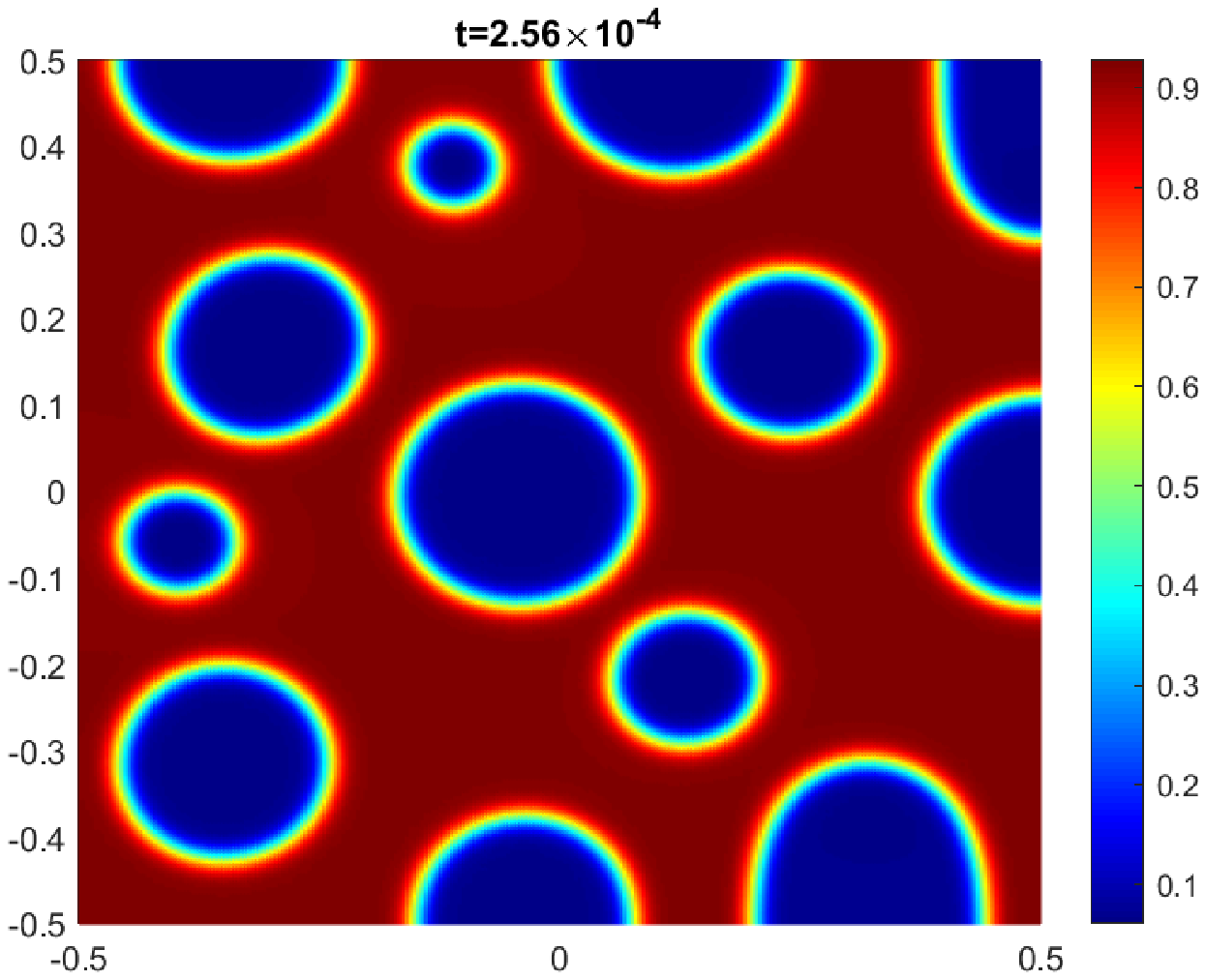}}
\subfigure{\includegraphics[width=0.325\textwidth]{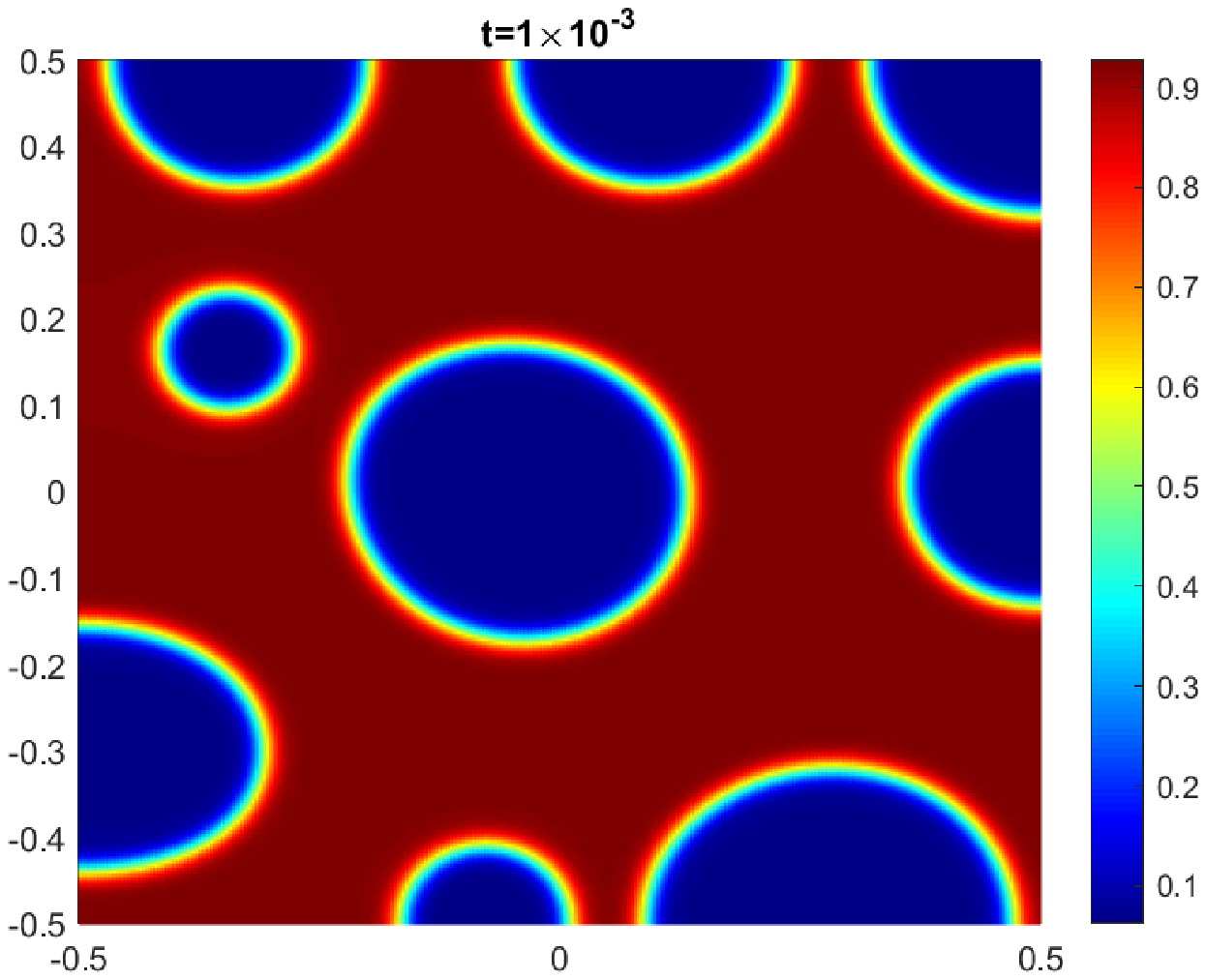}}
\subfigure{\includegraphics[width=0.325\textwidth]{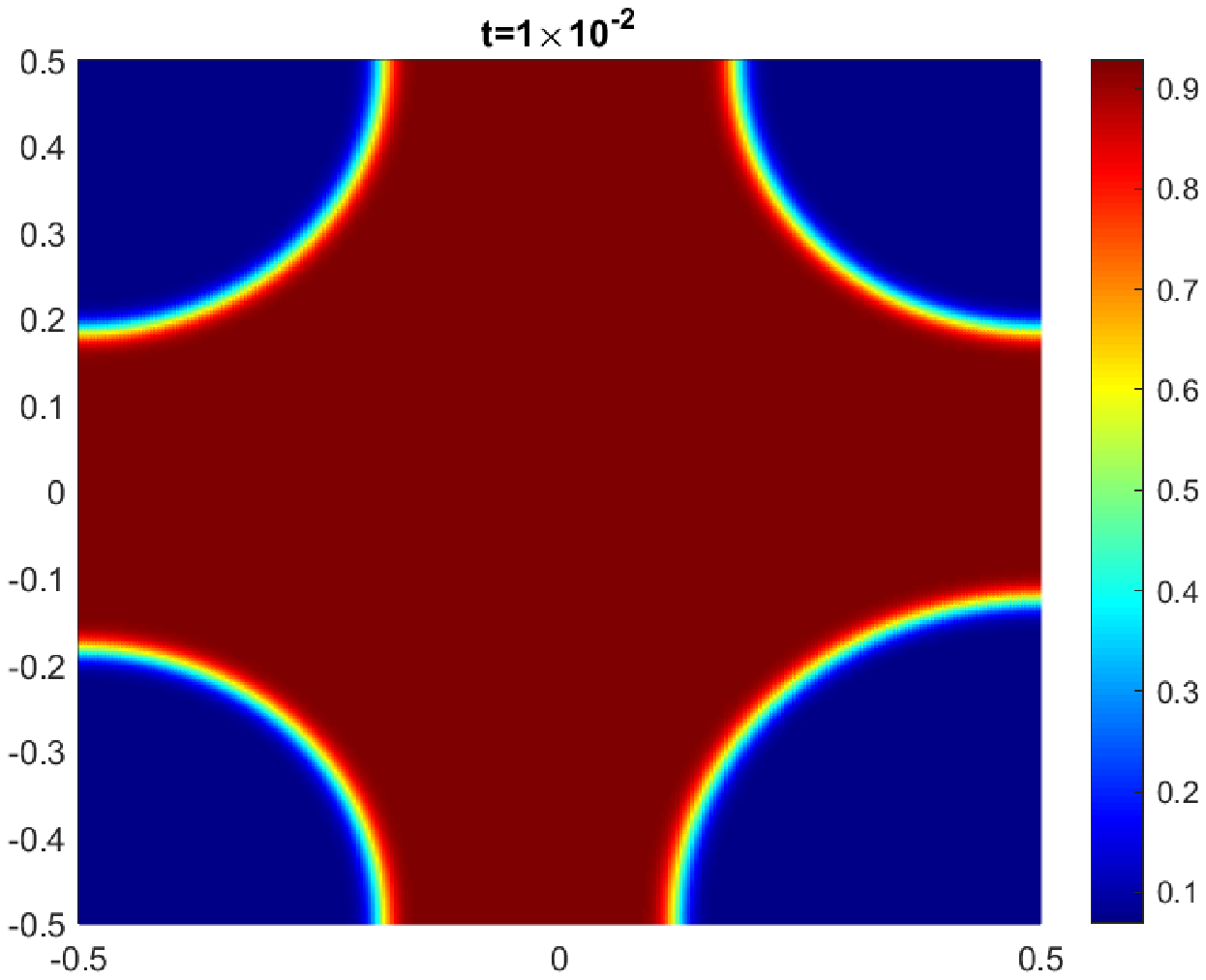}}
\caption{ The contours evolution of the numerical solution for the scheme (\ref{FPDGFull+}). } \label{Npat}
\end{figure}

\begin{figure}
\centering
\subfigure{\includegraphics[width=0.49\textwidth]{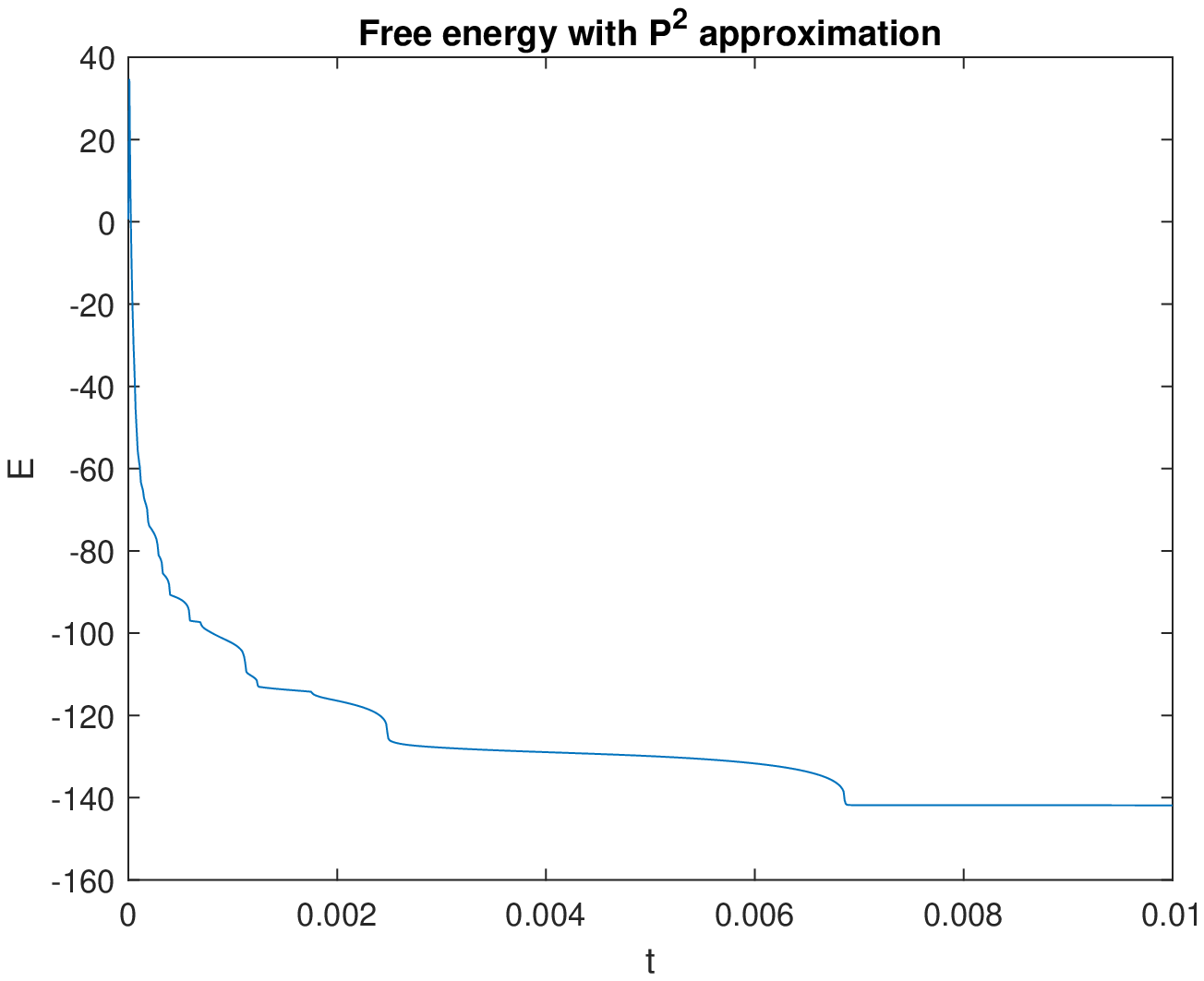}}
\subfigure{\includegraphics[width=0.49\textwidth]{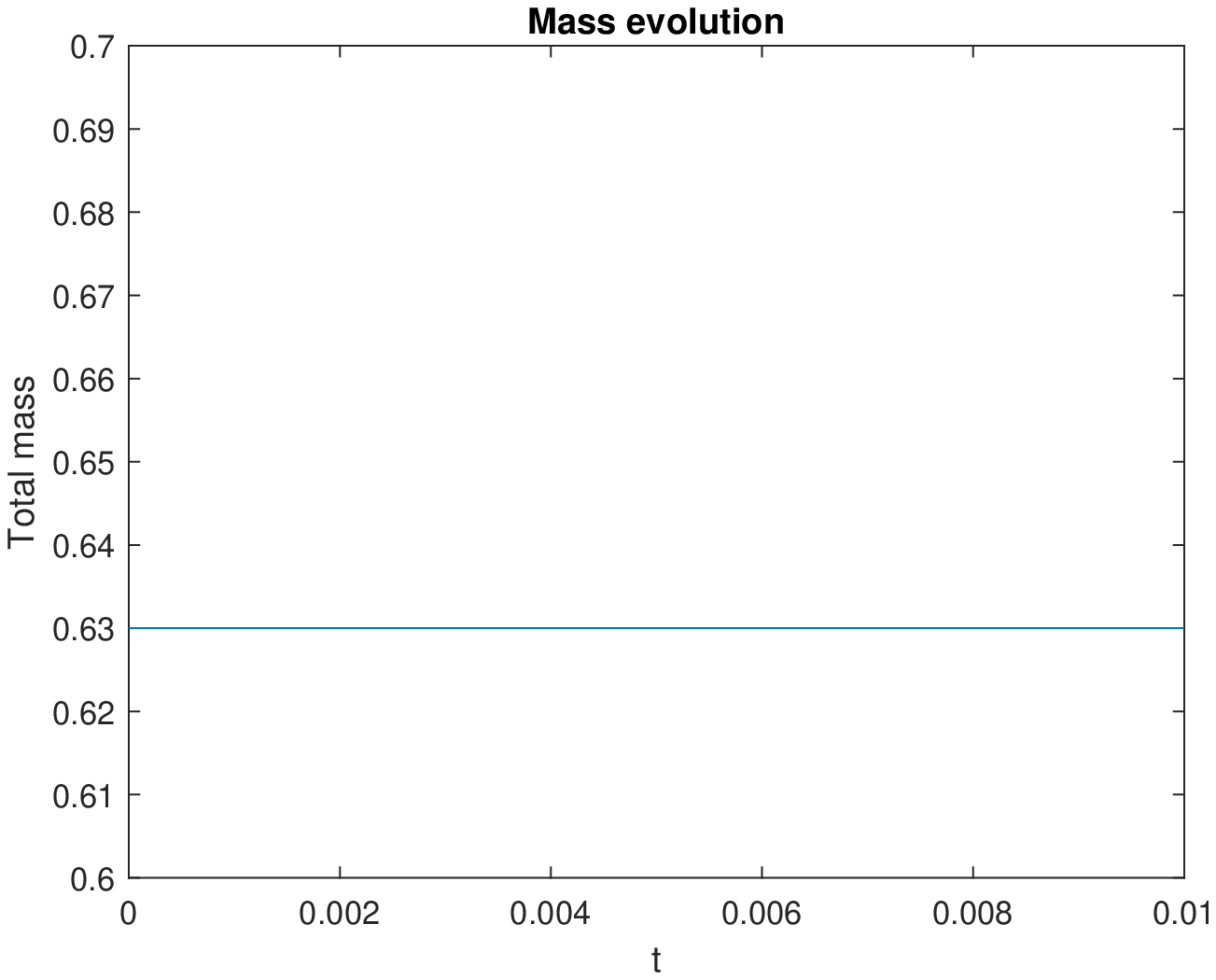}}
\caption{The energy and mass evolution of numerical solution for the scheme (\ref{FPDGFull+}).}\label{Nmasseng}
\end{figure}

\end{example}

\section{Conclusion}
In this paper, we integrate the mixed DG method with the IEQ method to design both first and second order fully discrete DG schemes that inherit the energy dissipation law and mass conservation of the continuous equation irrespectively of the mesh and time steps. The spatial discretization is based on the mixed DG method, and the temporal discretization is based on the IEQ approach introduced in \cite{Y16} for treating nonlinear potentials. Coupled with a spatial projection, the resulting IEQ-DG algorithms are easy to implement without resorting to any iteration method, and proven to be unconditionally energy stable and mass conservative. We have presented numerical examples to verify our theoretical results, and demonstrated the good performance of the scheme in terms of efficiency, accuracy, and preservation of solution properties such as  energy dissipation and mass conservation.

\section*{Acknowledgments}
This work was partially supported by the National Science Foundation under Grant DMS1812666.

\end{document}